\documentclass[a4paper,10pt]{article}
\usepackage[utf8]{inputenc}
\usepackage[english]{babel}

\usepackage{amsmath, amsfonts, amsthm, amssymb, amstext}
\usepackage{dsfont}
\usepackage{bbm,bm}
\usepackage{enumerate}

\usepackage{tikz}
\usetikzlibrary{arrows.meta}
\usepackage{graphicx}
\usepackage{url}
\usepackage{xcolor}
\usepackage{rotating}
\usepackage{hyperref}
\hypersetup{
   linktoc=page,
   colorlinks=true,
   linkcolor=blue,
   citecolor=blue,
   urlcolor=blue,
}

\usepackage[authoryear,round]{natbib}

\theoremstyle{plain}
\newtheorem{theorem}{Theorem}[section]
\newtheorem{lemma}[theorem]{Lemma}
\newtheorem{proposition}[theorem]{Proposition}
\newtheorem{thmdef}[theorem]{Theorem/Definition}

\newtheorem{corollary}[theorem]{Corollary}

\theoremstyle{definition}
\newtheorem{definition}[theorem]{Definition}
\newtheorem{example}[theorem]{Example}

\theoremstyle{remark}
\newtheorem{remark}[theorem]{Remark}

\newcommand{\RR}{\mathbb{R}}
\newcommand{\NN}{\mathbb{N}}
\newcommand{\PP}{\mathbb{P}}
\newcommand{\Sph}{\mathbb{S}}
\newcommand{\EE}{\mathbb{E}}

\title{Stochastic ordering in multivariate extremes}
\author{Michela Corradini\footnote{School of Mathematics, Cardiff University, Email: \texttt{CorradiniM@cardiff.ac.uk}} {}
and Kirstin Strokorb\footnote{School of Mathematics, Cardiff University, Email: \texttt{StrokorbK@cardiff.ac.uk}}}
\date{\today}

\begin{document}

\begin{minipage}{\linewidth}
$\phantom{abc}$
\vspace*{-1cm}
\maketitle
\end{minipage}

\begin{abstract}
The article considers the multivariate stochastic orders of upper orthants, lower orthants and positive quadrant dependence (PQD) among simple max-stable distributions and their exponent measures.
It is shown for each order that it holds for the max-stable distribution if and only if it holds for the corresponding exponent measure. The finding is non-trivial for upper orthants (and hence PQD order). From dimension $d\geq 3$ these three orders are not equivalent and a variety of phenomena can occur. However, every simple max-stable distribution PQD-dominates the corresponding independent model and is PQD-dominated by the fully dependent model. Among parametric models
the asymmetric Dirichlet family and the H\"usler-Rei{\ss} family turn out to be PQD-ordered according to the natural order within their parameter spaces.  For the H\"usler-Rei{\ss} family this holds true even for the supermodular order.
\end{abstract}

{\small
  \noindent \textit{Keywords}: Choquet model; concordance; exponent measure; lower orthant; majorisation; max-stable distribution; max-zonoid; multivariate distribution; positive quadrant dependence; portfolio; stable tail dependence function; upper orthant
\smallskip\\
  \noindent \textbf{2020 MSC} {60G70; 60E15}\\  
}

\section{Introduction}

Research on stochastic orderings and inequalities cover several decades, culminating among a vast literature for instance in the two monographs of \citet{shsh_07} and \citet{mueller_02}, the latter with a view towards applications and stochastic models, which appear in queuing theory, survival analysis, statistical physics or portfolio optimisation.
\citet{li_13} summarises developments of stochastic orders in reliability and risk management.
While the scientific activities  in finance, insurance, welfare economics or management science have been a driving force for many advances in the area, applications of stochastic orders are numerous and not limited to these fields. 
Importantly, such orderings will often play a role for robust inference, when only partial knowledge about a highly complex stochastic model is available.

Within the Extremes literature, related notions of positive dependence are well-known. It is a long-standing result that multivariate extreme value distributions exhibit positive association \citep{mo_83}. More generally, max-infinitely divisible distributions have this property as shown in \citet{resnick_08}, while \citet{beirlantetalii_03} summarise further implications in terms of positive dependence notions.
Recently, an extremal version of the popular MTP2 property (\citet{karlin_80, uhler_17})
has been studied in the context of multivariate extreme value distributions, especially H\"usler-Rei{\ss} distributions, and linked to graphical modelling, sparsity and implicit regularisation in multivariate extreme value models \citep{rez_23}.
Without any hope of being exhaustive, further fundamental scientific activity of the last decade on comparing stochastic models with a focus on multivariate extremes includes for instance an ordering of multivariate risk models based on extreme portfolio losses \citep{mainik_12}, inequalities for mixtures on risk aggregation
\citep{chen_22},
a comparison of dependence in multivariate extremes via tail orders
\citep{lichapter_13} or stochastic ordering for conditional extreme value modelling \citep{paptawn_15}.

\citet{yust_14} use stochastic dominance results from \citet{strsch_15} in order to derive bounds on the maximum portfolio loss and extend their work in \citet{yustco_20} to a distributionally robust inference for extreme {V}alue-at-{R}isk.

In this article we go back to some fundamental questions concerning stochastic orderings among multivariate extreme value distributions. We focus on the order of positive quadrant dependence (PQD order, also termed concordance order), which is defined via orthant orders.
Answers are given to the following questions. 
\begin{itemize}
\itemsep0mm
    \item \textit{What is the relation between orders among max-stable distributions and corresponding orders among their exponent measures?} (Theorem~\ref{thm:OO-characterization} and Corollary~\ref{cor:PQD})
    \item \textit{Can we find characterisations in terms of other typical dependency descriptors (stable tail dependence function, generators, max-zonoids)?} (Theorem~\ref{thm:OO-characterization})
    \item \textit{What is the role of fully independent and fully dependent model in this framework?} (Corollary~\ref{cor:PQD-boundaries}) 
      \item \textit{What is the role of Choquet/Tawn-Molchanov models in this framework?} (Corollary~\ref{cor:PQD-associated-Choquet} and Lemma~\ref{lemma:Choquet-LO-UO})
\end{itemize}
For lower orthants, the answers are readily deduced from standard knowledge in Extremes. It is dealing with the upper orthants that makes this work interesting. The key element in the proof of our most fundamental characterisation result, Theorem~\ref{thm:OO-characterization}, is based on Proposition~\ref{prop:key-OO-fundamentals} below, which may be of independent interest.
 Stochastic orders are typically considered for probability distributions only. In order to make sense of the first question, we introduce corresponding orders for exponent measures, which turn out natural in this context, cf.~Definition~\ref{def:orders-Lambda}.

\def\dirFigureWidth{0.27\linewidth}

\begin{figure}[bht]
    \centering \small
 \begin{tabular}{ccc}
    \includegraphics[width=\dirFigureWidth]{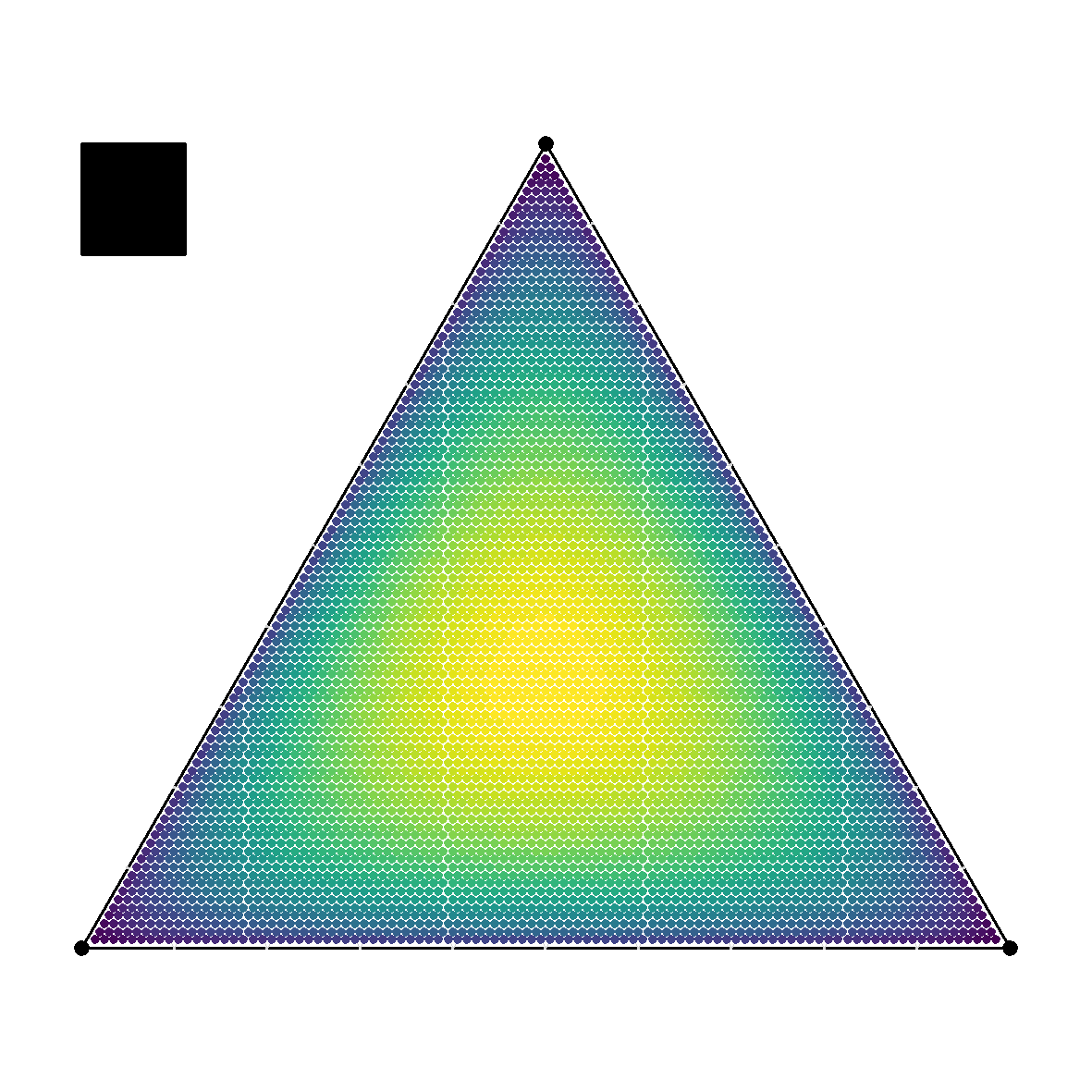}
    & \includegraphics[width=\dirFigureWidth]{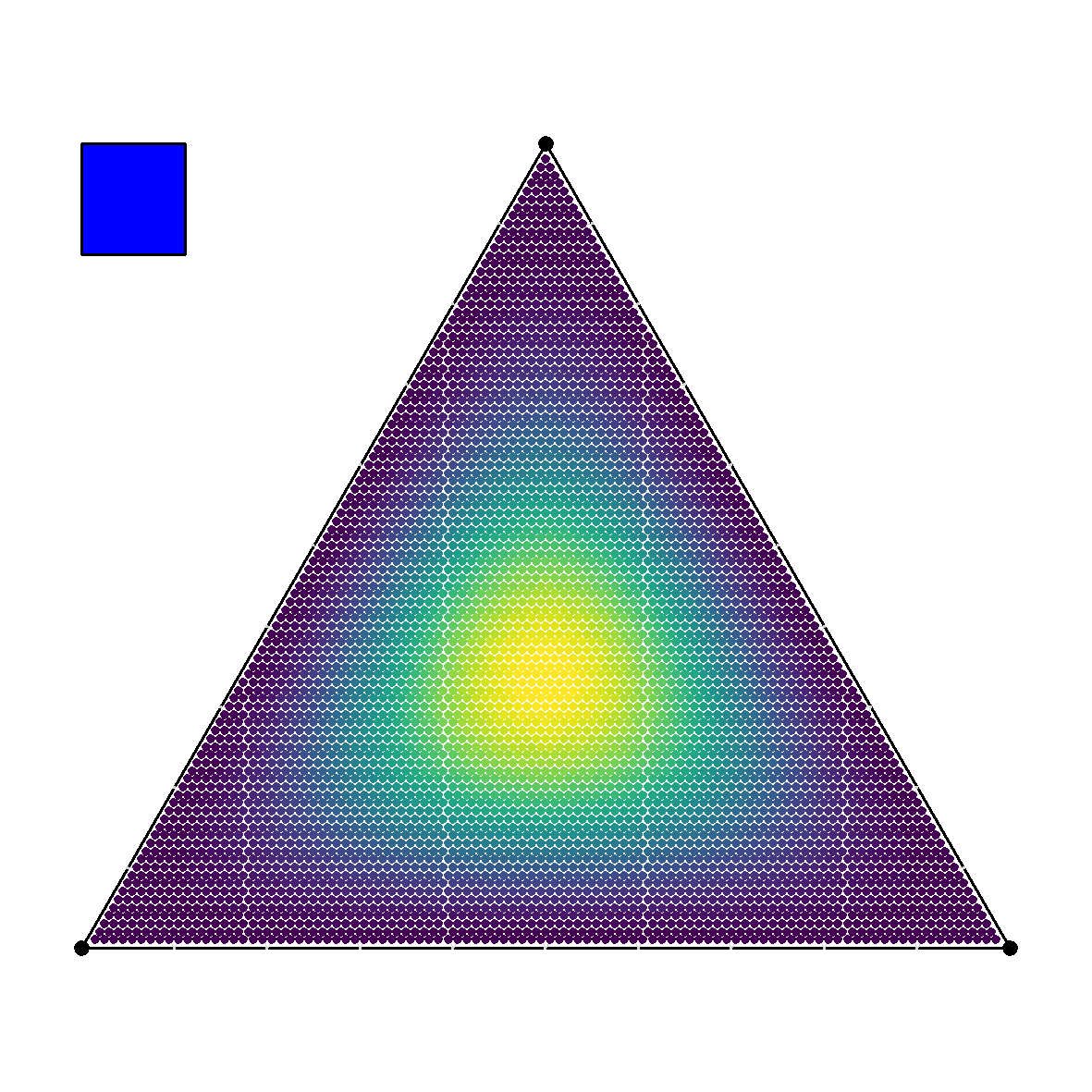}
    & \includegraphics[width=\dirFigureWidth]{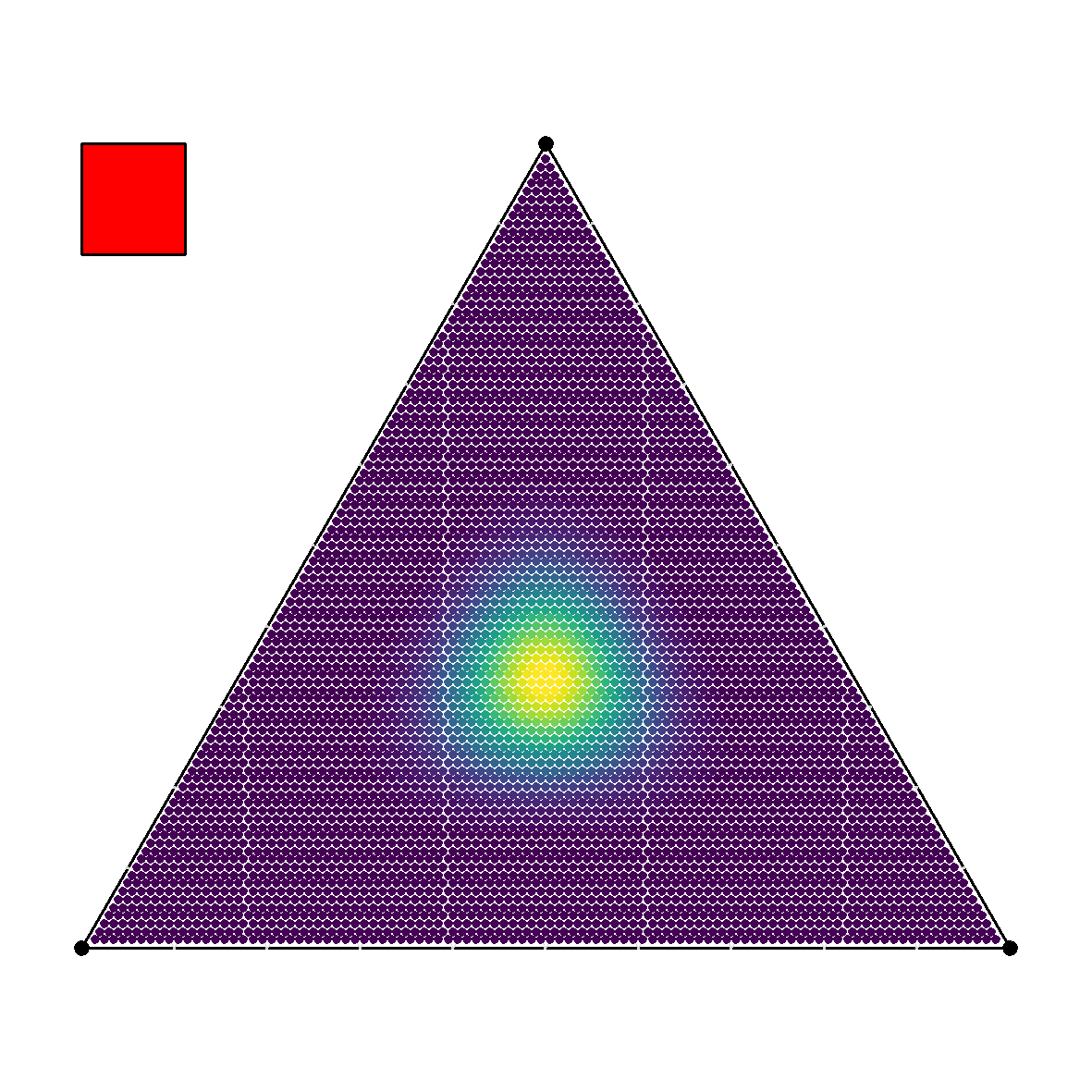}\\[-4mm]
 $\alpha=1.5$ & $\alpha=3$ & $\alpha=12$\\ 
    \includegraphics[width=\dirFigureWidth]{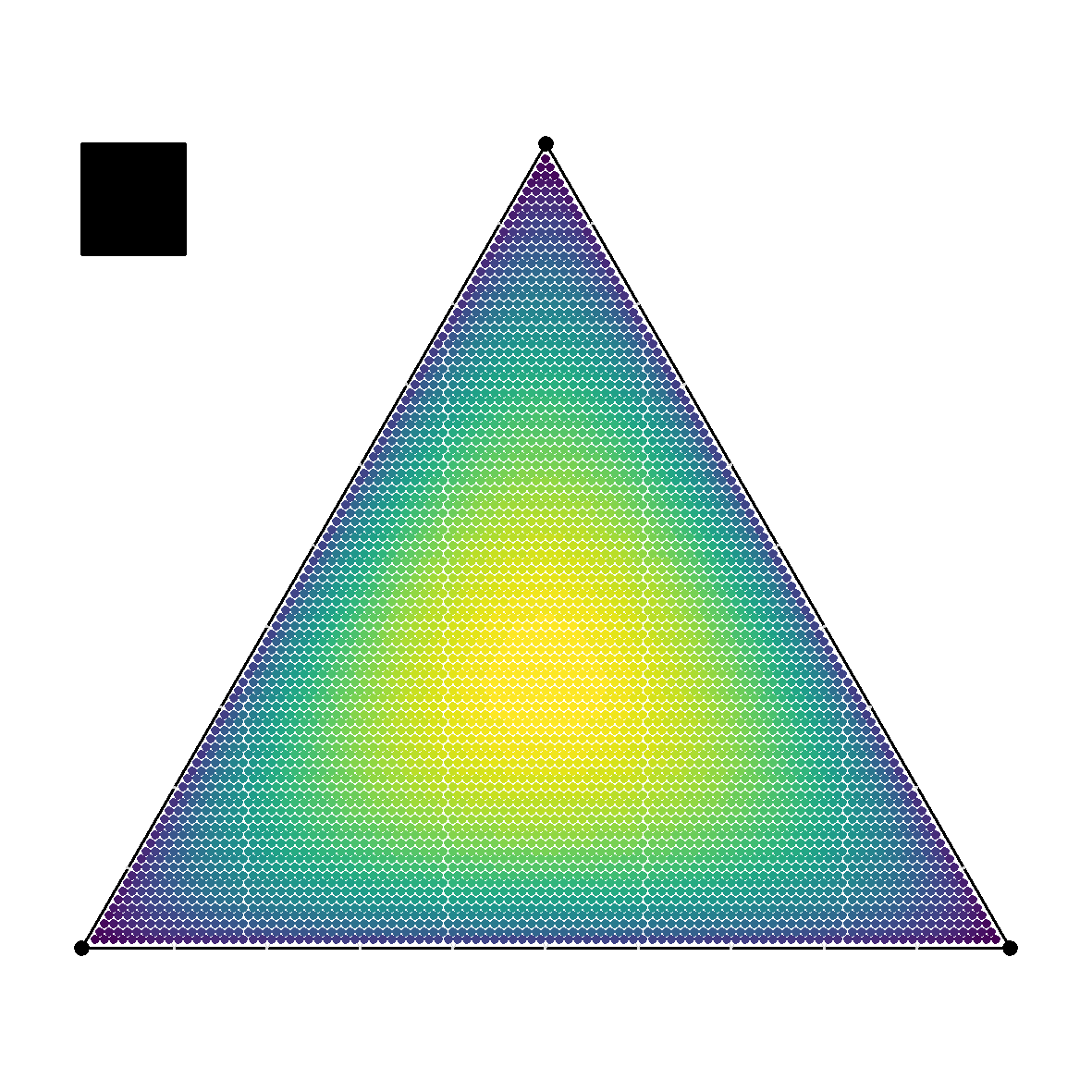}
    & \includegraphics[width=\dirFigureWidth]{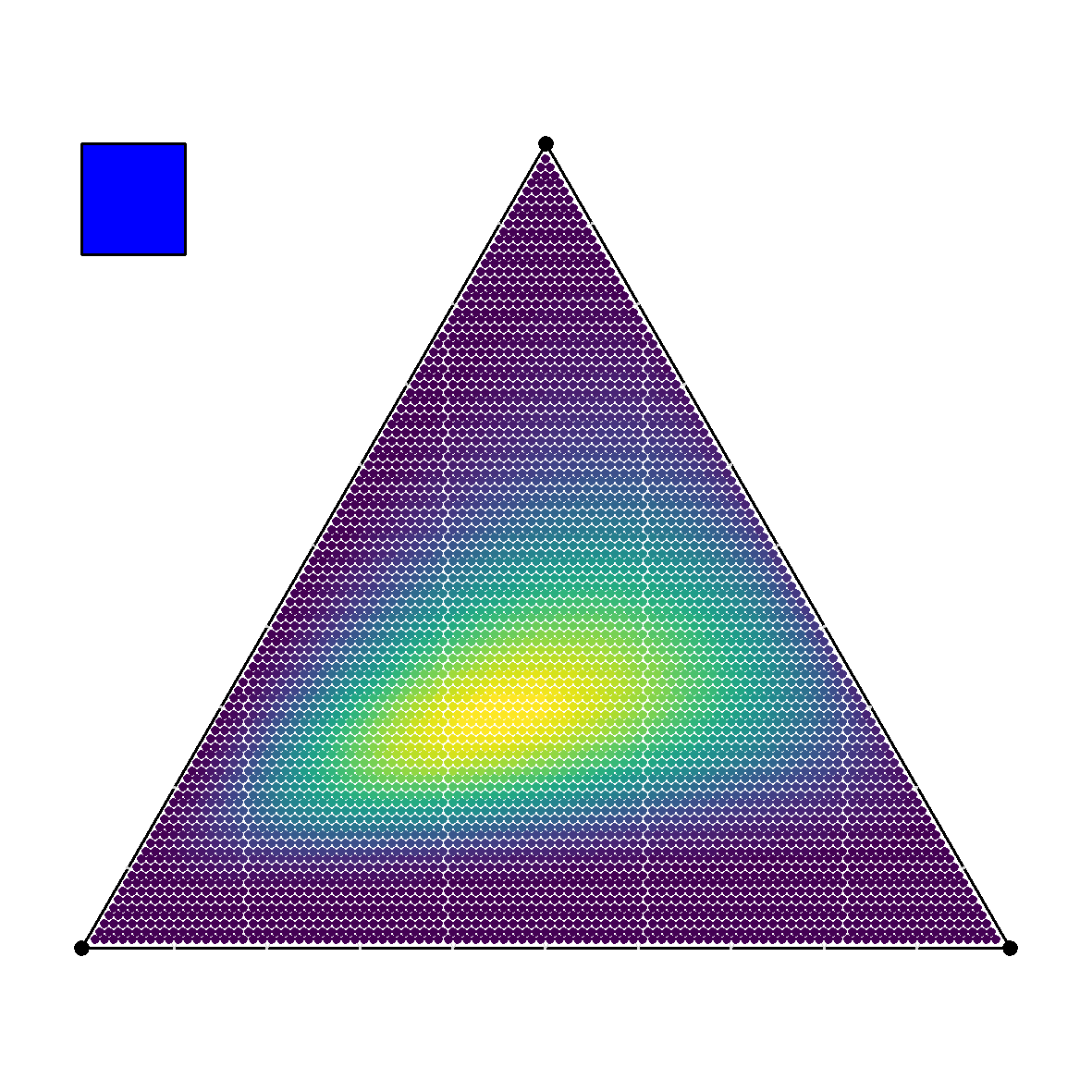}
    & \includegraphics[width=\dirFigureWidth]{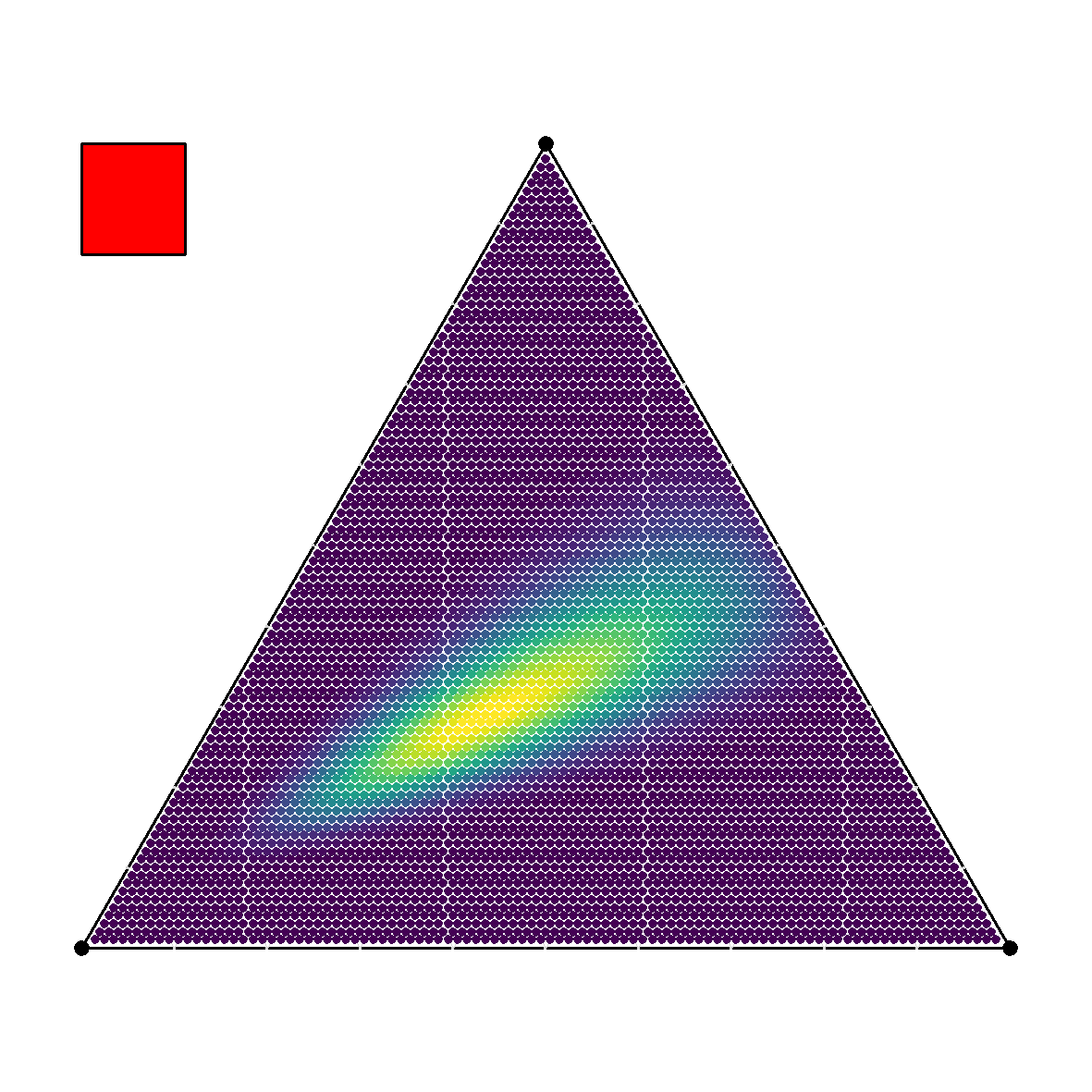}\\[-4mm]
    ${\bm \alpha}=(1.5,1.5,1.5)$ & ${\bm \alpha}=(1.5,3,12)$ & ${\bm \alpha}=(1.5,12,96)$\\
 \end{tabular}
    \caption{\small Angular densities (heat maps) of the symmetric max-stable Dirichlet model (top) and of the asymmetric max-stable Dirichlet model (bottom), cf.~\eqref{eq:DirichletAngularDensity} for an expression of the density. Larger values are represented by brighter colours.  The corresponding max-stable distributions are stochastically ordered in the PQD sense, increasing from left to right (Theorem~\ref{thm:Dirichlet-PQD}). The black, blue and red boxes encode the matching with Figures~\ref{fig:MinCDFs} and \ref{fig:MaxCDFs}.}
    \label{fig:DirichletAngularDensities}
\end{figure}

Second, we draw our attention to two popular parametric families of multivariate extreme value distributions that are closed under taking marginal distributions.
\begin{itemize}
\itemsep0mm
      \item \textit{Can we find order relations among the Dirichlet and H\"usler-Rei{\ss} parametric models?} (Theorem~\ref{thm:Dirichlet-PQD} and Theorem~\ref{thm:HRm})
\end{itemize}
The answers are affirmative.
For the H\"usler-Rei{\ss} model the result may be even strengthened for the supermodular order, which is otherwise beyond the scope of this article.
To give an impression of the result for the Dirichlet family, Figure~\ref{fig:DirichletAngularDensities} depicts six angular densities of the trivariate max-stable Dirichlet model. \citet{afz_15} showed already that the symmetric models associated with the top row densities are decreasing in the lower orthant sense. Our new result covers the asymmetric case depicted in the bottom row; we show that the associated multivariate extreme value distributions are decreasing in the (even stronger) PQD-sense (with a more streamlined proof).

Accordingly, our text is structured as follows. Section~\ref{sec:maxstability} recalls some basic representations of multivariate max-stable distributions and examples of parametric models that are relevant for subsequent results.  In Section~\ref{sec:order} we review the relevant multivariate stochastic orderings together with important closure properties. This section  contains also our (arguably natural) definition for corresponding order notions for exponent measures.  All main results are then given in Section~\ref{sec:results}.  Proofs and auxiliary results are postponed to Appendices~\ref{app:proofs} and~\ref{app:harmonic}.
Appendix~\ref{app:envelope} contains background material how we obtained the illustrations (max-zonoid envelopes for bivariate H\"usler-Rei{\ss} and Dirichlet families)  depicted in 
 Figures~\ref{fig:TMdepsets}, \ref{fig:Dirichletdepsets}, \ref{fig:AsymDirichletdepsets} and~\ref{fig:HRdepsets}.

\section{Prerequisites from multivariate extremes}
\label{sec:maxstability}

Our main results concern stochastic orderings among max-stable distributions, or, equivalently, orderings among their respective exponent measures, cf.~Theorem~\ref{thm:OO-characterization} below. Therefore, this section reviews some basic well-known results from the theory of multivariate extremes. Second, we will take a closer look at three marginally closed parametric families, the Dirichlet family, the H\"usler-Rei{\ss} family and the Choquet (Tawn-Molchanov) family of max-stable distributions, each model offering a different insight into phenomena of orderings among multivariate extremes.

\subsection{Max-stable random vectors and their exponent measures}

In order to clarify our terminology, we recall some definitions and basic facts about representations for  max-stable distributions, cf.~also \citet{resnick_08} or \citet{beirlantetalii_03}. Operations and inequalities between vectors are meant componentwise. We abbreviate ${\bm 0} = (0,0,\dots,0)^\top \in \RR^d$.
A random vector ${\bm X}=(X_1,\dots,X_d)^\top \in\RR^d$
is called \textit{max-stable} if for all $n \geq 1$ there exist suitable norming vectors ${\bm a_n}>{\bm 0}$ and ${\bm b_n}\in \RR^d$, such that the distributional equality
\[ \max_{i=1,\dots,n} ({\bm X_i}) \overset{d}{=}  {\bm a_n} {{\bm X} + {\bm b_n}} \] 
holds, where  ${\bm X_1},\dots,{\bm X_n}$ are  i.i.d.\ copies of ${\bm X}$.
According to the Fisher-Tippet theorem the marginal distributions $G_i(x)=\PP(X_i \leq x)$ are univariate max-stable distributions, that is, either degenerate to a point mass or a generalised extreme value (GEV) distribution of the form $G_{\xi}((x-\mu)/\sigma)$ with
$G_{\xi}(x)=\exp (-(1+\xi x)_+^{-{1}/{\xi}})$ if $\xi\ne0$ and $G_{0}(x)= \exp (-e^{-x})$,
where $\xi \in \RR$ is a shape-parameter, while $\mu\in \RR$ and $\sigma>0$ are the location and scale parameters, respectively. We write $\mathrm{GEV}(\mu,
\sigma,\xi)$ for short.

A max-stable random vector ${\bm X}=(X_1,\dots,X_d)^\top$ is called \textit{simple max-stable} if it has standard unit Fr\'echet marginals, that is, 
$\PP(X_i \leq x)=\exp (-1/x)$, $x>0$, for all $i=1,\dots,d$.
Any max-stable random vector ${\bm X}$ with GEV margins $X_i \sim \mathrm{GEV}(\mu_i,\sigma_i,\xi_i)$ can be transformed into a simple max-stable random vector ${\bm X^*}$ and vice versa via the componentwise order-preserving transformations 
\begin{align}\label{eq:marginaltransformation}
    {X_i^*}=\Big(1 + \xi_i \frac{X_i-\mu_i}{ \sigma_i}\Big)^{1/\xi_i}
    \quad \text{and} \quad X_i=\sigma_i \frac{( X^*_i)^{\xi_i}-1}{\xi_i}+{\mu_i}, \quad i=1,\dots,d
\end{align}
(with the usual interpretation of $(1+\xi x)^{1/\xi}$ as $e^x$ for $\xi=0$).
In this sense simple max-stable random vectors can be interpreted as a copula-representation for max-stable random vectors with non-degenerate margins, which encapsulates its dependence structure.

There are different ways to describe the distribution of such simple max-stable random vectors. The following will be relevant for us. 
Note that such vectors take values in the open upper orthant $(0,\infty)^d$ almost surely.
Here and hereinafter we shall denote the $i$-th indicator vector by ${\bm e_i}$ (all components of ${\bm e_i}$ are zero except for the $i$-th component, which takes the value one).

\begin{thmdef}[Representations of simple max-stable distributions] \label{td:representations}
 A random vector ${\bm X}=(X_1,\dots,X_d)^\top $ with distribution function 
$G({\bm x})=\exp(-V({\bm x}))$, ${\bm x}\in (0,\infty]^d$,
is simple max-stable if and only if the \emph{exponent function} $V$ can be represented in one of the following equivalent ways:
\begin{enumerate}[(i)]
\itemsep0mm
    \item \textbf{Spectral representation} \citep{dehaan_84}. There exists a finite measure space $(\Omega,\mathcal{A},\nu)$ and a measurable function $f: \Omega \to [0,\infty)^d$ such that $\int_{\Omega} f_i({\omega}) \; \nu(\mathrm{d}{\omega})=1$ for $i=1,\dots,d$, and \[ V(x_1,\dots,x_d)=\int_{\Omega} \max_{i=1,\dots,d}\frac{f_i(\omega)}{x_i} \; \nu(\mathrm{d}\omega). \]
      \item \textbf{Exponent measure} \citep{resnick_08}. There exists a $(-1)$-homogeneous measure $\Lambda$ on  $[0,\infty)^d \setminus \{{\bm 0}\}$, such that 
      \[\Lambda\Big( \big\{y \in [0,\infty)^d \,:\, y_i > 1\big\}\Big) = 1\]
for  $i=1,\dots,d$, and
\[V(x_1,x_2,\dots,x_d) = \Lambda\Big( \big\{y \in [0,\infty)^d \,:\, y_i > x_i \text{ for some } i \in \{1,\dots,d\} \big\}\Big).\]
      \item \textbf{Stable tail dependence function} \citep{ressel_13}. There exists a $1$-homogeneous and max-completely alternating function $\ell: [0,\infty)^d \to [0,\infty)$, such that $\ell({\bm e_i})=1$ for $i=1,\dots,d$, and \[ V(x_1,\dots,x_d) = \ell\bigg(\frac{1}{x_1},\dots,\frac{1}{x_d}\bigg)\]
      (cf.~Appendix~\ref{app:harmonic} for the notion of max-complete alternation).
\end{enumerate}
\end{thmdef}

In fact, the spectral representation can be seen as a polar decomposition of the {exponent measure} $\Lambda$, cf.~e.g.~\citet{resnick_08} or \citet{beirlantetalii_03}. Importantly, it is not uniquely determined by the law of ${\bm X}$.
Typical choices for the measure space $(\Omega,\mathcal{A},\nu)$ include (i) the unit interval with Lebesgue measure or (ii) a sphere $\Omega=\lbrace{\bm \omega}\in [0,\infty)^d \; : \; \lVert {\bm \omega} \rVert=1\rbrace$ with respect to some norm $\lVert \cdot \rVert$, for instance the $\ell_p$-norm
\[\lVert {\bm x} \rVert_p=\bigg(\sum_{i=1,\dots,d} \lvert x_i \rvert^p\bigg)^{1/p}\]
for some $p\geq 1$. 
For (i) it is then the spectral map $f$ which contains all dependence information. For (ii) one usually considers the component maps $f_i({\bm\omega})=\omega_i$, so that the measure $\nu$, then often termed \emph{\textbf{angular measure}}, absorbs the dependence information. 
For a given spectral representation $(\Omega,\mathcal{A},\nu,f)$ one may rescale $\nu$ to a probability measure and absorb the rescaling constant into the spectral map $f$. The resulting  random vector ${\bm Z}=(Z_1,\dots,Z_d)^\top$ such that $\EE(Z_i)={1}$, $i=1,\dots,d$, and \[ V(x_1,\dots,x_d)=\EE \max_{i=1,\dots,d}\bigg(\frac{Z_i}{x_i}\bigg),\]  has been termed \emph{\textbf{generator}} of ${\bm X}$, cf.~\citet{falk_19}.
A useful observation is the following; for a given vector ${\bm x}$ with values in $\RR^d$ and a subset $A \subset \{1,\dots,d\}$, let ${\bm x}_A$ be the subvector with components in $A$.
\begin{lemma}\label{lem:marZ}
    Let ${\bm Z}$ be a generator for the max-stable law ${\bm X}$, then  ${\bm Z}_A$ is a generator for ${\bm X}_A$.
\end{lemma}

   The stable tail dependence function $\ell$ goes back to \citet{huang_92} and has also been called \emph{D-norm} \citep{falk_04} of ${\bm X}$.
Since $\ell$ is 1-homogeneous, it suffices to know its values on the unit simplex $\triangle_d=\lbrace {\bm x}\in[0,\infty)^d \,:\, \lVert {\bm x} \rVert_1 =1 \rbrace$; the restriction of $\ell$ to $\triangle_d$ is called \emph{\textbf{Pickands dependence function}}
\[A(x_1,\dots,x_d) = \ell(x_1,\dots,x_d), \qquad (x_1,\dots,x_d)^\top  \in \triangle_d.\]
There exist further descriptors of the dependence structure, e.g.\ in terms of Point processes or LePage representation, cf. e.g.\ \citet{resnick_08} or, in a very general context, \citet{dmz2008}. Copulas of max-stable random vectors on standard uniform margins are called \emph{extreme value copulas} \citep{GudenSegers_10}. 

Let us close with a representation that allows for some interesting geometric interpretations. \citet{molchanov_08} introduced a convex body $K\subset [0,\infty)^d$, which can be interpreted (up to rescaling) as selection expectation of a random cross polytope associated with the (normalised) spectral measure $\nu$. It turns out that 
the stable tail dependence function is in fact the support function of $K$
\begin{align}\label{eq:ell-from-K}
  \ell({\bm x}) = \sup \lbrace \langle {\bm x},{\bm k}\rangle \; :\; {\bm k}\in K\rbrace.  
\end{align}
The convex body $K$ is called \emph{\textbf{max-zonoid}} (or \emph{\textbf{dependency set}}) of ${\bm X}$ and it is uniquely determined by the law of  ${\bm X}$. In fact
\begin{align}\label{eq:K-from-ell}
    K = \big\{{\bm k} \in [0,\infty)^d  \,:\, \langle {\bm k}, {\bm x} \rangle \leq \ell({\bm x}) \text{ for all } {\bm x} \in [0,\infty)^d \big\}. 
\end{align}

In general, it is difficult to translate one representation from Theorem~\ref{td:representations} into another apart from the obvious relations
\begin{align*}
    \ell({\bm x}) &=  \EE \max_{i=1,\dots,d} (x_iZ_i) = 
\int_{\Omega} \max_{i=1,\dots,d} x_i {f_i(\omega)}\; \nu(\mathrm{d}\omega)= \Lambda\Big( \big\{y \in [0,\infty)^d \,:\, \max_{i=1,\dots,d}(x_iy_i) > 1 \big\}\Big)
\end{align*}
for  ${\bm x} \geq {\bm 0}$. For convenience, we have added material in Appendix~\ref{app:envelope} how to obtain the boundary of a max-zonoid $K$ from the stable tail dependence function $\ell$ in the bivariate case, which will help to illustrate some of the results below.

\subsection{Parametric models}

Several parametric models for max-stable random vectors have been summarised for instance in \citet{beirlantetalii_03}. 
In what follows we draw our attention to two of the most popular parametric models, the Dirichlet and H\"usler-Rei{\ss} families, as well as the Choquet model (Tawn-Molchanov model), which will reveal some interesting phenomena and (counter-)examples of stochastic ordering relations.

\subsubsection{Dirichlet model} 

\citet{colestawn91} compute densities of angular measures of simple max-stable random vectors constructed from non-negative functions on the unit simplex $\triangle_d$.
In particular, the following \emph{asymmetric Dirichlet model} has been introduced. We summarise some equivalent characterisations, each of which may serve as a definition of the asymmetric Dirichlet model. 
This model has gained popularity due to its flexibility and simple structure forming the basis of Dirichlet mixture models \citep{boldi_07,sabourin_14}.

\begin{thmdef}[Multivariate max-stable Dirichlet distribution]\label{thmdef:Dirichlet}
A random vector ${\bm X}=(X_1,\dots,X_d)^\top$ is simple \emph{max-stable Dirichlet} distributed with parameter vector ${\bm\alpha}=(\alpha_1,\dots,\alpha_d)^\top\in(0,\infty)^d$, we write 
\[{\bm X}=(X_1,\dots,X_d)^\top ~\sim \mathrm{MaxDir}(\alpha_1,\dots,\alpha_d)=\mathrm{MaxDir}({\bm \alpha})\]
for short,  if and only if one of the following equivalent conditions is satisfied:
\begin{enumerate}[(i)]
\itemsep0mm
    \item \textbf{(Gamma generator)} A generator of ${\bm X}$ is the random vector 
    \[{\bm \alpha}^{-1}{\bm \Gamma}=(\Gamma_1/\alpha_1, \Gamma_2/\alpha_2,\dots,\Gamma_d/\alpha_d)^\top,\] 
    where
 ${\bm \Gamma}=(\Gamma_1,\dots,\Gamma_d)^\top$ consists of independent Gamma distributed variables $\Gamma_i\sim \Gamma(\alpha_i)$, $\alpha_i>0$, $i=1,\dots,d$. Here, the Gamma distribution $\Gamma(\alpha_i)$ has the density \[\gamma_{\alpha_i}(x)=\frac{ x^{\alpha_i-1}}{\Gamma(\alpha_i)}\exp\big(-{x}\big).\]
    \item \textbf{(Dirichlet generator)} A generator of ${\bm X}$ is the random vector 
    
    \[({\bm\alpha}^{-1}\lVert{\bm\alpha}\rVert_1){\bm D}
    =(\alpha_1+\dots+\alpha_d) \cdot (D_1/\alpha_1,D_2/\alpha_2,\dots,D_d/\alpha_d)^\top,\] 
    where
  ${\bm D}$ follows a Dirichlet distribution $\mathrm{Dir}(\alpha_1,\dots,\alpha_d)$ on the unit simplex $\triangle_d$ with density \[d({\omega_1,\dots,\omega_d})=\Gamma(\lVert{\bm\alpha}\rVert_1)\prod_{i=1}^d\frac{\omega_i^{\alpha_i-1}}{\Gamma(\alpha_i)}, \quad (\omega_1,\dots,\omega_d)^\top \in\triangle_d.\]
    \item \textbf{(Angular measure)} The density of the angular measure of ${\bm X}$ on $\triangle_d$ is given by 
    \begin{align}\label{eq:DirichletAngularDensity}
        h(\omega_1,\dots,\omega_d)=\frac{\Gamma(\lVert{\bm\alpha}\rVert_1+1)}{\lVert{\bm\alpha}{\bm\omega}\rVert_1}\prod_{i=1}^d\frac{\alpha_i^{\alpha_i}\omega_i^{\alpha_i-1}}{\Gamma(\alpha_i)(\lVert{\bm\alpha}{\bm\omega}\rVert_1)^{\alpha_i}}, \quad (\omega_1,\dots,\omega_d)^\top \in\triangle_d.
    \end{align}
\end{enumerate}
\end{thmdef}

To the best of our knowledge the representation through the Gamma generator, albeit inspired by \citet{afz_15} from the fully symmetric case, is new in this generality. We have added a proof in Appendix~\ref{app:DirichletHR}.
An advantage of the representation with the Gamma generator is that it reveals immediately the closure of the model with respect to taking marginal distributions, cf.~Lemma~\ref{lem:marZ}, a result that has been previously obtained in \citet{ballanischlather11}, but with a one-page proof and some intricate density calculations.

\begin{lemma}[Closure of Dirichlet model under taking marginals]
\label{lem:MaxDirClosure}
Let ${\bm X} = (X_1,\dots,X_d)^\top \sim \mathrm{MaxDir}(\alpha_1,\dots,\alpha_d)=\mathrm{MaxDir}({\bm \alpha})$  and $A \subset \{1,\dots,d\}$, then    ${\bm X}_A \sim \mathrm{MaxDir}({\bm \alpha}_A)$.
\end{lemma}

The angular density representation on the other hand is useful to see that different parameter vectors ${\bm \alpha} \neq {\bm \beta}$ lead in fact to different multivariate distributions $\mathrm{MaxDir}({\bm \alpha}) \neq \mathrm{MaxDir}({\bm \beta})$ for $d\geq 2$, so that $(0,\infty)^d$ is indeed the natural parameter space for this model.

\subsubsection{H\"usler-Rei{\ss} model}

The multivariate \emph{H\"usler-Rei{\ss} distribution} \citep{husler_89} forms the basis of the popular Brown-Resnick process \citep{kabluchko_09} and has sparked significant interest from the perspectives of spatial modelling \citep{davhustib19} and more recently in connection with graphical modelling of extremes \citep{enhitz20}.  
The natural parameter space for this model is the convex cone of  conditionally negative symmetric $d\times d$-matrices, whose diagonal entries are zero
\[
\mathcal{G}_d = \bigg\{ {\bm \gamma}=(\gamma_{ij})_{i,j \in \{1,\dots,d\}} \in \RR^{d \times d} \,:\, 
\begin{array}{c}
\gamma_{ij}=\gamma_{ji},\, \gamma_{ii}=0 \text{ for all $i,j \in \{1,\dots,d\}$, } \\
v^\top {\bm \gamma} v \leq 0 \text{ for all } v \in \RR^d \text{ such that } v_1+\dots+v_d=0 
\end{array}
\bigg\}.
\]
It is well-known, cf.~e.g.~\citet[Ch.~3]{bcr84}, that for a given ${\bm \gamma} \in \mathcal{G}_d$, there exists a zero mean Gaussian random vector ${\bm W}=(W_1,\dots,W_d)^\top$  with incremental variance
\begin{align}\label{eq:variogram}
    \EE (W_i-W_j)^2=\gamma_{ij}, \qquad i,j \in \{1,\dots,d\},
\end{align}
although its distribution is not uniquely specified by this condition. For instance, select $i \in \{1,\dots,d\}$. Imposing additionally the linear constraint
``$W_i=0$ almost surely'' leads to ${\bm W} \sim \mathcal{N}({\bm 0}, {\bm \Sigma_i})$ with
\begin{align*}
    ({\bm \Sigma_i})_{jk} = \frac{1}{2} \big( \gamma_{ij} + \gamma_{ik} - \gamma_{jk}\big),  \qquad j,k \in \{1,\dots,d\},
\end{align*}
which satisfies \eqref{eq:variogram}.

\begin{thmdef}[Multivariate H\"usler-Rei{ss} distribution, cf.~\citet{kabluchko_11} Theorem~1] \label{thmdef:HR}
Let ${\bm \gamma} \in \mathcal{G}_d$ and \eqref{eq:variogram} be valid.
Consider the simple max-stable random vector ${\bm X}=(X_1,\dots,X_d)^\top$ defined by the generator ${\bm Z}=(Z_1,\dots,Z_d)^\top$ with
\[Z_i = \exp\bigg(W_i - \frac{1}{2} \mathrm{Var}(W_i)\bigg), \qquad i=1,\dots,d.\]
Then the distribution of ${\bm X}$ depends only on ${\bm \gamma}$ and not on the specific choice of a zero mean Gaussian distribution satisfying \eqref{eq:variogram}. 
We call ${\bm X}$  simple \emph{H\"usler-Rei{\ss}} distributed with parameter matrix ${\bm \gamma}$ and write for short
\[{\bm X}=(X_1,\dots,X_d)^\top ~\sim \mathrm{HR}({\bm \gamma}).\]
\end{thmdef}

We also note that for ${\bm \gamma}_1, \, {\bm \gamma}_2 \in \mathcal{G}_d$, the distributions $\mathrm{HR}({\bm \gamma}_1)$ and $\mathrm{HR}({\bm \gamma}_2)$ coincide if and only if ${\bm \gamma}_1 = {\bm \gamma}_2$, so that $\mathcal{G}_d$ is indeed the natural parameter space for these models. This follows directly from the observation that the multivariate H\"usler-Rei{\ss} model is also closed under taking marginal distributions and the equivalent statement for bivariate H\"usler-Rei{\ss} models, which can be seen for instance from \eqref{eq:ellHR} below. Indeed, we also state the following lemma for clarity. It follows directly from the generator representation of $\mathrm{HR}({\bm \gamma})$ and Lemma~\ref{lem:marZ}.

\begin{lemma}[Closure of H\"usler-Rei{\ss} model under taking marginals]
Let ${\bm X} = (X_1,\dots,X_d)^\top \sim \mathrm{HR}({\bm \gamma})$ and $A \subset \{1,\dots,d\}$, then    ${\bm X}_A \sim \mathrm{HR}({\bm \gamma}_{A \times A})$, where ${\bm \gamma}_{A \times A}$ is the restriction of ${\bm \gamma}$ to the components of $A$ in both rows and columns.
\end{lemma}

It is well-known that up to a change of location and scale parameters H\"usler-Rei{\ss} distributions are the only possible limit laws of maxima of triangular arrays of multivariate Gaussian distributions, a finding which can be traced back to \citet{husler_89} and \citet{brownresnick77}.
The following version will be convenient for us.

\begin{theorem}[Triangular array convergence of maxima of Gaussian vectors, cf.~\citet{kabluchko_11} Theorem~2] 
\label{thm:triangularArrayHR}
Let $u_n$ be a sequence such that $\sqrt{2\pi} u_n e^{u_n^2/2}/n \to 1$ as $n \to \infty$. For each $n \in \NN$ let ${\bm Y^{(n)}_1},{\bm Y^{(n)}_2},\dots, {\bm Y^{(n)}_n}$ be independent copies of a $d$-variate zero mean unit-variance Gaussian random vector with correlation matrix $(\rho^{(n)}_{ij})_{i,j \in \{1,\dots,d\}}$.
Suppose that for all $i,j \in \{1,\dots,d\}$
\[
4 \log(n) (1-\rho^{(n)}_{ij}) \to \gamma_{ij} \in [0,\infty) 
\]
as $n \to \infty$. Then the matrix 
    ${\bm \gamma}=(\gamma_{ij})_{i,j \in \{1,\dots,d\}}$ is necessarily and element of $\mathcal{G}_d$. Let ${\bm M}^{(n)}$ be the componentwise maximum of ${\bm Y^{(n)}_1},{\bm Y^{(n)}_2},\dots, {\bm Y^{(n)}_n}$. Then the componentwise rescaled vector $u_n({\bm M}^{(n)}-u_n)$ converges in distribution to the H\"usler-Rei{\ss} distribution $\mathrm{HR}({\bm \gamma})$.
\end{theorem}

\begin{remark} \label{ref:HRextended}
    In the bivariate case we have $\gamma_{12}=\gamma_{21}=\gamma \in [0,\infty)$ and the boundary case $\gamma=0$ leads to a degenerate random vector with fully dependent components, whereas $\gamma \uparrow \infty$ leads to a random vector with independent components. More generally, one might also admit the value $\infty$ for $\gamma_{ij}$ in the multivariate case, as long as the resulting matrix ${\bm \gamma}$ is \emph{negative definite in the extended sense}, cf.~\citet{kabluchko_11}. This extension corresponds to a partition of ${\bm X}$ into independent  subvectors ${\bm X} = \bigsqcup_A {\bm X}_A$, where each ${\bm X}_A$ is a H\"usler-Rei{\ss} random vector in the usual sense. Here $\gamma_{ij}=\infty$ precisely when $i$ and $j$ are in different subsets of the partition.  
    Theorem~\ref{thm:triangularArrayHR} extends to this situation as well. In fact, is has been formulated in this generality in \citet{kabluchko_11}.
\end{remark}

\subsubsection{Choquet model / Tawn-Molchanov model}
\label{sec:choquet}

A popular way to summarise extremal dependence information within a random vector is by considering its \emph{extremal coefficients},
which in the case of a simple max-stable random vector ${\bm X}=(X_1,X_2,\dots,X_d)^\top$ 
may be expressed as
\[\theta(A)=\ell({\bm e}_A), 
\qquad 
{\bm e}_A = \sum_{i \in A} {\bm e}_i, \qquad 
A \subset \{1,\dots,d\}, \, A \neq \emptyset,\]
or, equivalently,
\begin{align}\label{eq:theta}
    \theta(A) = \EE \max_{i \in A} Z_i = 
\int_{\Omega} \max_{i \in A} {f_i(\omega)}\; \nu(\mathrm{d}\omega)= \Lambda\Big( \big\{y \in [0,\infty)^d \,:\, \max_{i \in A}(y_i) > 1 \big\}\Big),
\end{align}
where $\ell$ is the stable tail dependence function, ${\bm Z}$ a generator, $\Lambda$ the exponent measure and $(\Omega,\mathcal{A},\nu,f)$ a spectral representation for ${\bm X}$. Loosely speaking, the coefficient $\theta(A)$, which takes values in $[1,|A|]$, can be interpreted as the effective number of independent variables among the collection $(X_i)_{i \in A}$. We have $\theta(\{i\})=1$ for singletons $\{i\}$ and naturally $\theta(\emptyset)=0$.

The following result can be traced back to \citet{schlatawn02} and \citet{molchanov_08}. 
Accordingly, the associated max-stable model, which can be parametrised by its extremal coefficients, has been introduced as \emph{Tawn-Molchanov model} in \citet{strsch_15}. It is essentially an application of the the Choquet theorem (see \citet{molchanovTheory} Section~1.2 and \citet{bcr84} Theorem~6.6.19), which also holds for not necessarily finite capacities (see \citet{schneiderweil2008} Theorem 2.3.2). Therefore, it has been relabelled  \emph{Choquet model} in \citet{mostr16}, cf.~Appendix~\ref{app:harmonic} for background on complete alternation. We write $\mathcal{P}_d$ for the power set of $\{1,\dots,d\}$ henceforth.

\begin{theorem} \label{thm:choquet}
\begin{enumerate}[a)]
\itemsep0mm
    \item Let $\theta: \mathcal{P}_d \to \RR$. Then $\theta$ is the extremal coefficient function of a simple max-stable random vector in $(0,\infty)^d$ if and only if $\theta(\emptyset)=0$, $\theta(\{i\})=1$ for all $i=1,\dots,d$ and $\theta$ is union-completely alternating.
    \item Let  $\theta: \mathcal{P}_d \to \RR$ be an extremal coefficient function. Let 
\[\ell^*({\bm x}) = \int_0^\infty \theta(\{i \,:\, x_i \geq t\}) \, dt, \quad {\bm x} \in [0,\infty)^d\]
be the Choquet integral with respect to $\theta$. Then $\ell^{*}$ is a valid stable tail dependence function, which retrieves the given extremal coefficients
$\ell^{*}({\bm e}_A)=\theta(A)$ for all $A \in \mathcal{P}_d$. Its max-zonoid is given by
\[K^*
= \big\{{\bm k} \in [0,\infty)^d  \,:\, \langle {\bm k}, {\bm e}_A \rangle \leq \theta(A) \text{ for all } A \in \mathcal{P}_d \big\}. \]

\item Let $\ell$ be any stable tail dependence function with extremal coefficient function $\theta$ and $K$ its corresponding max-zonoid. Then
\[\ell({\bm x}) \leq \ell^*({\bm x}), \quad {\bm x} \geq 0 \qquad \text{and} \qquad K \subset K^*. \]
\end{enumerate}
\end{theorem}

\begin{example}[Choquet model in the bivariate case]
    Let $\ell$ be a  bivariate stable tail dependence function and $\theta=\ell(1,1)\in [1,2]$ the bivariate extremal coefficient. Then the associated Choquet model is given by the max-zonoid $K^*= \{(x_1,x_2) \in [0,1]^2\,:\, x_1+x_2 \leq \theta\}$ or the stable tail dependence function $\ell^*(x_1,x_2)=\max(x_1 + (\theta-1) x_2, (\theta-1) x_1 + x_2)$. Figure~\ref{fig:TMdepsets} displays a situation, where the original $\ell$ stems from an asymmetric Dirichlet model.
\end{example}

\begin{figure}[htb]
    \centering
    \includegraphics[width=0.3\textwidth]{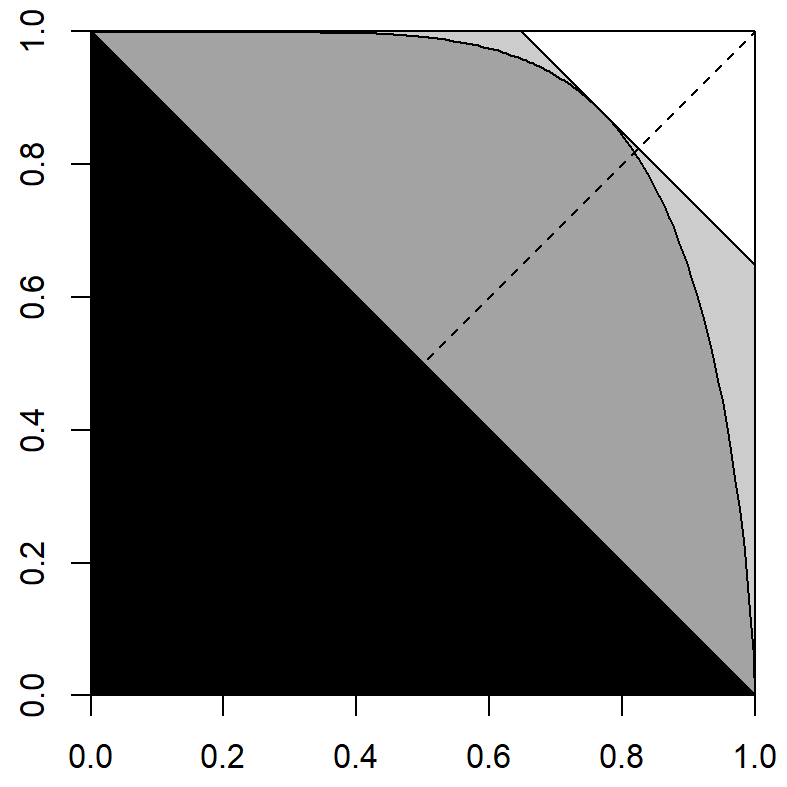}
    \includegraphics[width=0.3\textwidth]{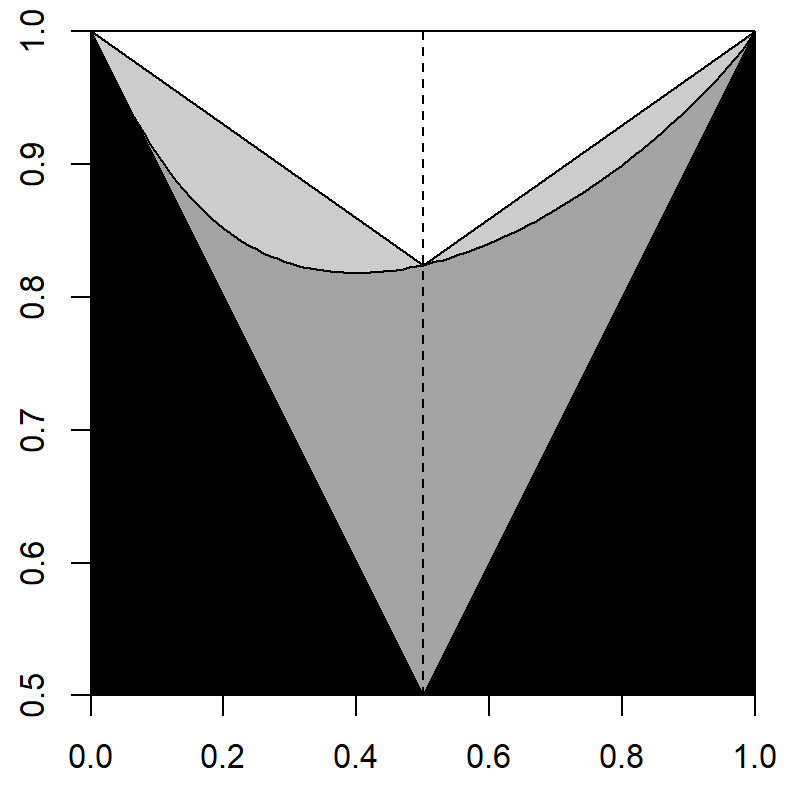}
    \caption{\small Nested max-zonoids and Pickands dependence functions ranging from full dependence (black), an asymmetric Dirichlet model with ${\bm \alpha}=(30,0.2)$ (dark grey), its associated Choquet (Tawn-Molchanov) model (light grey) to the fully independent model (white).}
    \label{fig:TMdepsets}
\end{figure}

In geometric terms, for any given max-zonoid $K \subset [0,1]^d$ the associated Choquet max-zonoid $K^* \subset [0,1]^d$ is bounded by $2^d-1$ hyperplanes, one for each direction ${\bm e}_A$, which is the supporting hyperplane of the max-zonoid $K$ in the direction of ${\bm e}_A$.

The Choquet model is a spectrally discrete max-stable model, whose exponent measure has its support contained in the rays through the vectors ${\bm e}_A$, $A \subset \{1,\dots,d\}$, $A \neq \emptyset$.
While its natural parameter space is the set of extremal coefficients, we can also describe it via the mass that the model puts on those rays. To this end, let $\tau: \mathcal{P}_d \setminus \{\emptyset\} \to \RR$ be given as follows 
\begin{align*}
   \tau(A) &= - (\Delta_{a_1} \Delta_{a_2} \dots \Delta_{a_n} \theta)(\{1,\dots,d\} \setminus A )=\sum_{I \subset A} (-1)^{|I|+1} \theta(I \cup (\{1,\dots,d\}\setminus A)) , 
\end{align*}
where we assume $a_1,a_2,\dots, a_n$ to be the distinct elements from $A \subset \{1,\dots,d\}$. Then the spectral representation $(\Omega,\mathcal{A},\nu^{*},f)$ with $\Omega=\{{\bm \omega} \in [0,\infty)^2 \,:\, \lVert {\bm \omega} \rVert_\infty = 1\}$, $f_i({\bm \omega})=\omega_i$ and 
\begin{align}\label{eq:Choquet-spectral}
\nu^{*} = \sum_{A \in \mathcal{P}_d \setminus \{\emptyset\}} \tau(A) \,\delta_{{\bm e}_A}
\end{align}
corresponds to the stable tail dependence function $\ell^{*}$ from Theorem~\ref{thm:choquet}.
In terms of an underlying generator for which \eqref{eq:theta} holds true, we may express $\tau$ as
\[\tau(A) = \EE \Big(\min_{i \in A} Z_i - \max_{i \in \{1,\dots,d\}\setminus A} Z_i\Big)_+,\]
cf.~\citet{papstr16} Lemma~3.
 Moreover, we recover $\theta$ from $\tau$ via
\[
\theta(A) = \sum_{K \,:\, K \cap A \neq \emptyset} \tau(K),
\] 
which makes the analogy between extremal coefficient functions $\theta$ and capacity functionals of random sets even more explicit. 

However, there are two drawbacks with representing the Choquet model by the collection of coefficients $\tau(A)$, $A \subset \{1,\dots,d\}$, $A \neq \emptyset$.
First, this representation is specific to the dimension, in which the model is considered, that is, we cannot simply turn to a subset of these coefficients when considering marginal distributions. Second, one may easily forget that one has in fact not $2^d-1$ degrees of freedom among these coefficients, but $2^d-1-d$, since $\theta(\{i\})=1$ for singletons $\{i\}$, which is only encoded through $d$ linear constraints for $\tau$ as follows
\begin{align}\label{eq:tauMarginal}
    \sum_{K \,:\, i \in K} \tau(K) = 1, \qquad i=1,\dots,d.
\end{align}

A third parametrisation of the Choquet model, which has received little attention so far, but is very relevant for the ordering results in this article (cf.~Lemma~\ref{lemma:Choquet-LO-UO}) and does not have such drawbacks, is the following. 
Instead of extremal coefficients, let us consider the following \emph{tail dependence coefficients} for $A \subset \{1,\dots,d\}$, $A \neq \emptyset$:
\[
\chi(A) = \EE \min_{i \in A} Z_i = 
\int_{\Omega} \min_{i \in A} {f_i(\omega)}\; \nu(\mathrm{d}\omega)= \Lambda\Big( \big\{y \in [0,\infty)^d \,:\, \min_{i \in A}(y_i) > 1 \big\}\Big).
\]
Then it is easily seen that
\begin{align}\label{eq:chi-theta}
    \chi(A)=\sum_{I \subset A, \, I \neq \emptyset} (-1)^{|I|+1} \theta(I) \qquad \text{and} \qquad  \theta(A)=\sum_{I \subset A, \, I \neq \emptyset} (-1)^{|I|+1} \chi(I).
\end{align} 
Since $\theta(\emptyset)=0$, and with $a_1,\dots,a_n$ being the distinct elements of $A$, the first identity may also be expressed as
\begin{align}\label{eq:chi-nonnegative}
\chi(A) = - (\Delta_{a_1} \Delta_{a_2} \dots \Delta_{a_n} \theta)(\emptyset).
\end{align}
In particular $\chi(\{i\})=\theta(\{i\})=1$ for $i=1,\dots,d$, and these operations show explicitly, how $\theta$ and $\chi$ can be recovered from each other. While $\theta$ resembles a capacity functional, $\chi$ can be seen as an analog of an inclusion functional, since 
\begin{align}\label{eq:chi-from-tau}
  \chi(A) = \sum_{K \,:\, A \subset K} \tau(K),  
\end{align}
whereas
\[\tau(A)  = \sum_{K \,:\, A \subset K} (-1)^{|K \setminus A|}\chi(K) = (\Delta_{b_1} \Delta_{b_2} \dots \Delta_{b_m} \chi)( A ),\]
where $b_1,b_2,\dots,b_m$ are the distinct elements of $\{1,2,\dots,d\} \setminus A$.

To sum up, we may consider three different parametrizations for the Choquet model: 
\begin{enumerate}[(i)]
\itemsep0mm
    \item by the $2^d-1$ extremal coefficients $\theta(A)$, $A \in \mathcal{P}_d$, $A \neq \emptyset$, 
    \item by the $2^d-1$ tail dependence coefficients  $\chi(A)$, $A \in \mathcal{P}_d$, $A \neq \emptyset$, 
    \item by the $2^d-1$ mass coefficients  $\tau(A)$, $A \in \mathcal{P}_d$, $A \neq \emptyset$.
\end{enumerate}
For (i) and (ii) the constraint for standard unit Fr\'echet margins is encoded via 
$\chi(\{i\})=\theta(\{i\})=1$ for $i=1,\dots,d$. For (iii) it amounts to the $d$ conditions from \eqref{eq:tauMarginal}. Only (i) and (ii) do not depend on the dimension, in which the model is considered.

\section{Prerequisites from stochastic orderings}
\label{sec:order}

A wealth of stochastic orderings and associated inequalities have been summarised in \citet{mueller_02} and \citet{shsh_07}, the most fundamental order being the \emph{usual stochastic order} \[F\leq_{\mathrm{st}} G\] between two univariate distributions $F$ and $G$, which is defined as $F(x)\geq G(x)$ for all $x\in \RR$. This means that draws from $F$ are less likely to attain large values than draws from $G$.

For multivariate distributions definitions of orderings are less straightforward and there are many more notions of stochastic orderings. We will focus on upper orthants, lower orthants and the PQD order here.
A subset $U \subset \RR^d$ is called an \emph{upper orthant} if it is of the form 
 \[U=U_{{\bm a}}=\lbrace{\bm x} \in \RR^d \, :\,  x_1 > a_1,\dots,x_d > a_d\rbrace\]
 for some  ${\bm a}\in \RR^d$. 
   Similarly, a subset $L \subset \RR^d$ is called a \emph{lower orthant} if it is of the form 
 \[L=L_{{\bm a}}=\lbrace{\bm x} \in \RR^d \, :\,  x_1 \leq a_1,\dots,x_d \leq a_d\rbrace\]
 for some  ${\bm a}\in \RR^d$.

\begin{definition}[Multivariate orders LO, UO, PQD, \citet{shsh_07}, Sections 6.G and 9.A, \citet{mueller_02}, Sections 3.3. and 3.8]
\label{def:UO-LO-PQD}
$\phantom{a}$\\
Let ${\bm X},{\bm Y}\in\RR^d$ be two random vectors.  
\begin{itemize}
\itemsep0mm
    \item ${\bm X}$ is said to be \emph{smaller than ${\bm Y}$ in the upper orthant order}, denoted ${\bm X}\leq_{\mathrm{uo}}{\bm Y}$,\\ if
$\PP({\bm X}\in U)\leq\PP({\bm Y}\in U)$ for all upper orthants $U \subset \RR^d$. 
\item ${\bm X}$ is said to be \emph{smaller than ${\bm Y}$ in the lower orthant order}, denoted ${\bm X}\leq_{\mathrm{lo}}{\bm Y}$,\\
if $\PP({\bm X}\in L)\geq\PP({\bm Y}\in L)$ for all lower orthants $L \subset \RR^d$. 
\item ${\bm X}$ is said to be \emph{smaller than ${\bm Y}$ in the positive quadrant order }, denoted ${\bm X}\leq_{\mathrm{PQD}}{\bm Y}$,\\ if we have the relations ${\bm X}\leq_{\mathrm{uo}}{\bm Y}$ and ${\bm X}\geq_{\mathrm{lo}}{\bm Y}$.
\end{itemize}
\end{definition}

Note that the PQD order (also termed concordance order) is a dependence order. If ${\bm X}\leq_{\mathrm{PQD}}{\bm Y}$ holds, it implies that ${\bm X}$ and ${\bm Y}$ have identical univariate marginals.
Several equivalent characterizations of these orders are summarised in the respective sections of \citet{mueller_02} and \citet{shsh_07}. In relation to portfolio properties, it is interesting to note that for non-negative random vectors ${\bm X}, {\bm Y} \in [0,\infty)^d$
\begin{align}
\label{eq:UO-nonnegative-min}
    {\bm X}\leq_{\mathrm{uo}}{\bm Y} \quad \iff \quad \min_{i=1,\dots,d}(a_i X_i) \leq_{\mathrm{st}} \min_{i=1,\dots,d}(a_i Y_i) \quad \text{ for all }  {\bm a} \in (0,\infty)^d;\\
    \label{eq:LO-nonnegative-max}
      {\bm X}\leq_{\mathrm{lo}}{\bm Y} \quad \iff \quad \max_{i=1,\dots,d}(a_i X_i) \leq_{\mathrm{st}} \max_{i=1,\dots,d}(a_i Y_i) \quad \text{ for all }  {\bm a} \in (0,\infty)^d.
\end{align}
In addition, if ${\bm X}, {\bm Y} \in [0,\infty)^d$ and ${\bm X}\leq_{\mathrm{lo}}{\bm Y}$,  then 
\[\EE \, g\bigg(\sum_{i=1}^d a_iX_i\bigg) \leq \EE \, g\bigg(\sum_{i=1}^d a_iY_i\bigg),
\] 
for all ${\bm a}\in [0,\infty)^d$ and all \emph{Bernstein functions} $g$, provided that the expectation exists, cf.~\citet{shsh_07} 6.G.14 and 5.A.4 for this fact and Appendix~\ref{app:harmonic} for a definition of Bernstein functions.  In particular, such functions are non-negative, monotonously increasing and concave and therefore form a natural class of utility functions, see e.g.\ \citet{brockett_87} and \citet{caballe_96}.  Important examples of Bernstein functions include the identity function, $g(x)=\log(1+x)$ or $g(x)=(1+x)^\alpha - 1$ for $\alpha \in (0,1)$.

The multivariate orders from Definition~\ref{def:UO-LO-PQD} have several useful \textbf{closure properties}. We refer to \citet{mueller_02} Theorem 3.3.19 and Theorem 3.8.7 for a systematic collection, including
\begin{itemize}
\itemsep0mm
    \item independent or identical concatenation,
    \item marginalisation,
    \item distributional convergence,
    \item applying increasing transformations to the components,
    \item taking mixtures. 
\end{itemize}

In what follows, we will need a corresponding notion of multivariate  orders not only for probability measures on $\RR^d$, but also for exponent measures as introduced in Section~\ref{sec:maxstability}. While the support of an exponent measure $\Lambda$ is contained in $[0,\infty)^d\setminus \{{\bm 0}\}$, its total mass is infinite. We only know for sure that $\Lambda(B)$ is finite for Borel sets $B\subset \RR^d$ \emph{bounded away from the origin} in the sense that there exists $\varepsilon>0$, such that $B \cap L_{\varepsilon {\bm e}} = \emptyset$ (recall 
$ L_{\varepsilon {\bm e}} = \lbrace{\bm x} \in \RR^d \, :\,  x_1 \leq \varepsilon,\dots,x_d \leq \varepsilon\rbrace$). This means that we need to assume a different view on lower orthants and work with their complements instead, a subtlety, which did not matter previously when defining such notions for probability measures only.
The following notion seems natural in view of Definition~\ref{def:UO-LO-PQD} and the results of Section~\ref{sec:results}. Figure~\ref{fig:orthants} illustrates the restriction to fewer admissible test sets for these orders for exponent measures.

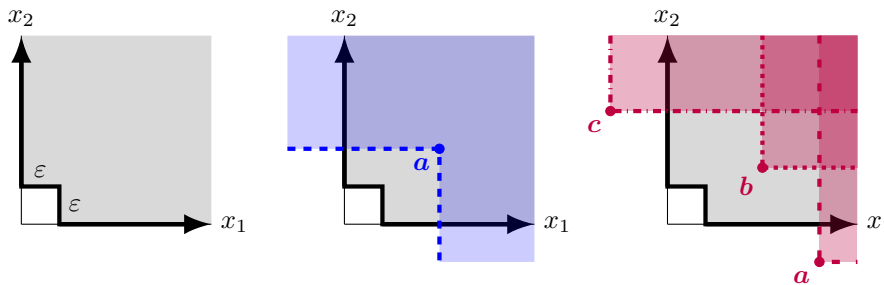
\begin{figure}[hbt]
    \centering
\begin{tikzpicture}[scale = 2.5, >=Latex]
    \coordinate (O) at (0,0);
    \coordinate (e1) at (1,0);
    \coordinate (e2) at (0,1);
    \coordinate (e12) at (1,1);
    \coordinate (e1m) at (0.9,0);
    \coordinate (e2m) at (0,0.9);
    \coordinate (eps1) at (0.2,0);
    \coordinate (eps2) at (0,0.2);
    \coordinate (eps12) at (0.2,0.2);
    \draw (O) to (e1) node[anchor=west]{$x_{1}$};
    \draw (O) to (e2) node[anchor=south]{$x_{2}$};
 \fill[opacity=.3,gray] (e2) -- (eps2) -- (eps12) -- (eps1) -- (e1) -- (e12);
    \draw[ultra thick, black] (e2m) -- (eps2) -- (eps12) -- (eps1) -- (e1m);
       \draw[ultra thick, black, ->] (eps1) to (e1);
    \draw[ultra thick, black, ->] (eps2) to (e2);
    \fill[opacity=0.2,blue] (1,1) -- (-0.3,1) -- (-0.3,0.4) -- (0.5,0.4) -- (0.5,-0.2) -- (1,-0.2);
     \draw[dashed,blue,ultra thick]  (-0.3,0.4) -- (0.5,0.4) -- (0.5,-0.2);
     \filldraw[blue] (0.5,0.4) circle[radius=0.7pt] node[anchor = north east] {${\bm a}$};
     \def\sh{1.7}
    \coordinate (O) at (0+\sh,0);
    \coordinate (e1) at (1+\sh,0);
    \coordinate (e2) at (0+\sh,1);
    \coordinate (e12) at (1+\sh,1);
    \coordinate (e1m) at (0.9+\sh,0);
    \coordinate (e2m) at (0+\sh,0.9);
    \coordinate (eps1) at (0.2+\sh,0);
    \coordinate (eps2) at (0+\sh,0.2);
    \coordinate (eps12) at (0.2+\sh,0.2);
    \draw (O) to (e1) node[anchor=west]{$x_{1}$};
    \draw (O) to (e2) node[anchor=south]{$x_{2}$};
 \fill[opacity=.3,gray] (e2) -- (eps2) -- (eps12) -- (eps1) -- (e1) -- (e12);
    \draw[ultra thick, black] (e2m) -- (eps2) -- (eps12) -- (eps1) -- (e1m);
       \draw[ultra thick, black, ->] (eps1) to (e1);
    \draw[ultra thick, black, ->] (eps2) to (e2);
    \fill[opacity=0.3,purple] (1+\sh,1) -- (-0.3+\sh,1) -- (-0.3+\sh,0.6) -- (1+\sh,0.6);
     \draw[dashdotted,purple,ultra thick] (-0.3+\sh,1) -- (-0.3+\sh,0.6) -- (1+\sh,0.6);
     \filldraw[purple] (-0.3+\sh,0.6) circle[radius=0.7pt] node[anchor = north east] {${\bm c}$};
    \fill[opacity=0.3,purple] (1+\sh,1) -- (0.8+\sh,1) -- (0.8+\sh,-0.2) -- (1+\sh,-0.2);
     \draw[loosely dashed,purple,ultra thick] (0.8+\sh,1) -- (0.8+\sh,-0.2) -- (1+\sh,-0.2); 
     \filldraw[purple] (0.8+\sh,-0.2) circle[radius=0.7pt] node[anchor = north east] {${\bm a}$};
    \fill[opacity=0.3,purple] (1+\sh,1) -- (0.5+\sh,1) -- (0.5+\sh,0.3) -- (1+\sh,0.3);
     \draw[dotted,purple,ultra thick] (0.5+\sh,1) -- (0.5+\sh,0.3) -- (1+\sh,0.3); 
     \filldraw[purple] (0.5+\sh,0.3) circle[radius=0.7pt] node[anchor = north east] {${\bm b}$};
     \def\sh{-1.7}
    \coordinate (O) at (0+\sh,0);
    \coordinate (e1) at (1+\sh,0);
    \coordinate (e2) at (0+\sh,1);
    \coordinate (e12) at (1+\sh,1);
    \coordinate (e1m) at (0.9+\sh,0);
    \coordinate (e2m) at (0+\sh,0.9);
    \coordinate (eps1) at (0.2+\sh,0);
    \coordinate (eps2) at (0+\sh,0.2);
    \coordinate (eps12) at (0.2+\sh,0.2);
    \draw (O) to (e1) node[anchor=west]{$x_{1}$};
    \draw (O) to (e2) node[anchor=south]{$x_{2}$};
 \fill[opacity=.3,gray] (e2) -- (eps2) -- (eps12) -- (eps1) -- (e1) -- (e12);
    \draw[ultra thick, black] (e2m) -- (eps2) -- (eps12) -- (eps1) -- (e1m);
       \draw[ultra thick, black, ->] (eps1) to (e1);
       \draw[opacity=0] (eps1) to node[midway,right,opacity=1]{$\varepsilon$} (eps12);
         \draw[opacity=0] (eps2) to node[midway,above,opacity=1]{$\varepsilon$} (eps12);
    \draw[ultra thick, black, ->] (eps2) to (e2);
\end{tikzpicture}    
    \caption{Illustration of test sets for multivariate orders for exponent measures in dimension $d=2$, cf.~Definition~\ref{def:orders-Lambda}. Left: $\Lambda$ is locally finite on the (closed) grey area for all $\varepsilon>0$, its total (infinite) mass is contained in the union of such sets; middle: admissible complement of a lower orthant $\RR^2\setminus L_{\bm a}$ (blue area) for testing lower orthant order for $\Lambda$; right: admissible upper orthants $U_{\bm a}$, $U_{\bm b}$, $U_{\bm c}$ (three red areas) for testing upper orthant order for $\Lambda$.}
    \label{fig:orthants}
\end{figure}

\begin{definition}[Multivariate orders for exponent measures]
\label{def:orders-Lambda}

Let $\Lambda,\widetilde \Lambda$ be two infinite measures on $\RR^d$ with mass contained in  $[0,\infty)^d\setminus \{{\bm 0}\}$ and taking finite values on Borel sets bounded away from the origin.
\begin{itemize}
\itemsep0mm
    \item $\Lambda$ is said to be \emph{smaller than $\widetilde \Lambda$ in the upper orthant order}, denoted $\Lambda \leq_{\mathrm{uo}}\widetilde \Lambda$,\\ if
$\Lambda(U)\leq \widetilde\Lambda(U)$ for each upper orthant $U \subset \RR^d$ that is bounded away from the origin. 
\item $\Lambda$ is said to be \emph{smaller than $\widetilde \Lambda$ in the lower orthant order}, denoted $\Lambda \leq_{\mathrm{lo}}\widetilde \Lambda$,\\
if $\Lambda(\RR^d \setminus L)\leq \widetilde \Lambda(\RR^d \setminus L)$ for all lower orthants $L \subset \RR^d$ such that $\RR^d \setminus L$ is bounded away from the origin. 
\item $\Lambda$ is said to be \emph{smaller than $\widetilde \Lambda$ in the positive quadrant order}, denoted $\Lambda \leq_{\mathrm{PQD}} \widetilde \Lambda$,\\ if we have the relations $\Lambda \leq_{\mathrm{uo}} \widetilde \Lambda$ and $\Lambda \geq_{\mathrm{lo}} \widetilde \Lambda$.
\end{itemize}
\end{definition}

\begin{remark}  Exponent measures $\Lambda$ and $\widetilde \Lambda$ are Radon measures on $[0,\infty]^d \setminus \{{\bm 0}\}$ (the one-point uncompactification of $[0,\infty]^d$). Any Borel set $B\subset [0,\infty]^d \setminus \{{\bm 0}\}$ bounded away from the origin is relatively compact in this space, hence  
$\Lambda(B)$ and  $\widetilde \Lambda(B)$, including  $\Lambda(U)$, $\widetilde \Lambda(U)$, $\Lambda(\RR^d \setminus L)$ and  $\widetilde \Lambda(\RR^d \setminus L)$ as above, are all finite.
\end{remark}

\section{Main results}\label{sec:results}

First we present some fundamental characterisations of LO, UO and PQD order among simple max-stable distributions  and their exponent measures, then we study these orders among the introduced parametric families.
While we focus on simple max-stable distributions in what follows, we would like to stress that applying componentwise identical isotonic transformations to random vectors preserves orthant and concordance orders; in this sense the following properties can be seen as statements about the respective copulas. In particular, among max-stable random vectors, it suffices to establish these orders among simple max-stable random vectors and they translate immediately to all counterparts with different marginal distributions, cf.~\eqref{eq:marginaltransformation}.

\subsection{Fundamental results}
\label{sec:fundamental-results}

We start by assembling the most fundamental relations for multivariate orders among simple max-stable random vectors.
While the statements about lower orthant orders are almost immediate from existing theory and definitions, the relations for upper orthants are a bit more intricate and non-standard in the area. In particular, showing that 	``$\Lambda \leq_{\mathrm{uo}} \widetilde \Lambda$ implies $G \leq_{\mathrm{uo}}  \widetilde G$''   turns out to be non-trivial. 
The key ingredient 
in the proof of the following theorem (cf.~Appendix~\ref{app:proofs-fundamentals}) is Proposition~\ref{prop:key-OO-fundamentals} for part b). 

\begin{theorem}[Orthant orders characterisations] 
\label{thm:OO-characterization}
Let $G$ and $\widetilde G$ be $d$-variate simple max-stable distributions with exponent measures $\Lambda$ and $\widetilde \Lambda$, generators ${\bm Z}$ and $\widetilde {\bm Z}$, stable tail dependence functions $\ell$ and $\widetilde \ell$ and max-zonoids $K$ and $\widetilde K$, respectively. 
\begin{enumerate}[a)]
    \item The following statements are equivalent.
    \begin{enumerate}[(i)]
        \itemsep0mm
        \item $G\leq_{\mathrm{lo}} \widetilde G$;
         \item $\Lambda\leq_{\mathrm{lo}} \widetilde \Lambda$;
         \item $\EE(\max_{i=1,\dots,d}(a_i Z_i)) \leq \EE(\max_{i=1,\dots,d}(a_i \widetilde Z_i))$ for all ${\bm a} \in (0,\infty)^d$;
         \item $\ell \leq \widetilde \ell$;
         \item $K \subset \widetilde K$.
    \end{enumerate}
    \item The following statements are equivalent.
       \begin{enumerate}[(i)]
        \itemsep0mm
        \item $G\leq_{\mathrm{uo}} \widetilde G$;
         \item $\Lambda\leq_{\mathrm{uo}} \widetilde \Lambda$; 
           \item $\EE(\min_{i \in A}(a_i Z_i)) \leq \EE(\min_{i \in A}(a_i \widetilde Z_i))$ for all ${\bm a} \in (0,\infty)^d$ and $A \subset \{1,\dots,d\}$, $A \neq \emptyset$. 
    \end{enumerate}
    \item If $d=2$, the following statements are equivalent.
    \begin{enumerate}[(i)]
        \itemsep0mm
        \item $G\leq_{\mathrm{PQD}} \widetilde G$;
        \item $G\leq_{\mathrm{uo}} \widetilde G$;
         \item $G\geq_{\mathrm{lo}} \widetilde G$.
    \end{enumerate}
\end{enumerate}
\end{theorem}

The assumption $d=2$ is important in part c); these equivalences are no longer true in higher dimensions, cf.~Example~\ref{example:UO-LO-not-equivalent} below. 
Theorem~\ref{thm:OO-characterization} implies further that the orthant ordering of two generators ${\bm Z}$ and ${\widetilde {\bm Z}}$  implies the respective ordering of the corresponding distributions ${G}$ and ${\widetilde G}$ and exponent measures ${\Lambda}$ and ${\widetilde \Lambda}$. However, the converse is false and most generators will not satisfy orthant orderings, even when the corresponding distributions do. An interesting case for this phenomenon is the H\"usler-Rei{\ss} family, cf.~Example~\ref{example:HR-generators-not-ordered} below. 
The following corollary is another immediate consequence of Theorem~\ref{thm:OO-characterization}.

\begin{corollary}[PQD/concordance order characterisation]
\label{cor:PQD}
Let $G$ and $\widetilde G$ be $d$-variate simple max-stable distributions with exponent measures $\Lambda$ and $\widetilde \Lambda$, then
\[G\leq_{\mathrm{PQD}} \widetilde G 
\quad \iff \quad 
\Lambda\leq_{\mathrm{PQD}} \widetilde \Lambda.
\]
\end{corollary}

 It is well-known that for any stable tail dependence function $\ell$ of a simple max-stable random vector 
\begin{align}\label{eq:trivialLO}
\ell_{\mathrm{dep}}({\bm x}) = \lVert {{\bm x}}\rVert_\infty \leq \ell({\bm x}) \leq \lVert {{\bm x}}\rVert_1=\ell_{\mathrm{indep}}({\bm x}),    \qquad {\bm x} \geq {\bm 0},
\end{align}
where $\ell_{\mathrm{dep}}$ represents the degenerate max-stable random vector, whose components are fully dependent, and $\ell_{\mathrm{indep}}$ corresponds to the max-stable random vector with completely independent components. 
From the perspective of stochastic orderings this means that every max-stable random vector is dominated by the fully independent model, while it dominates the fully dependent model with respect to the lower orthant order. It seems less well-known that the converse ordering holds true for upper orthants, so that we arrive at the following corollary.

\begin{corollary}[PQD/concordance for independent and fully dependent model] \label{cor:PQD-boundaries}
    Let $G_{\mathrm{indep}}$, $G_{\mathrm{dep}}$ and $G$ be $d$-dimensional simple max-stable distributions, where  $G_{\mathrm{indep}}$ represents the model with fully independent components, and $G_{\mathrm{dep}}$ represents the model with fully dependent components. Then
    \[G_{\mathrm{indep}} \, \leq_{\mathrm{PQD}}\,  G \,\leq_{\mathrm{PQD}}\, G_{\mathrm{dep}}.
\]
\end{corollary}

Similarly Theorem~\ref{thm:choquet} can be strengthened as follows. Whilst previously only the lower orthant order was known, we have in fact PQD/concordance ordering.

\begin{corollary}[PQD/concordance for the associated Choquet model] \label{cor:PQD-associated-Choquet}
    Let ${\bm X}$ be a simple max-stable random vector with extremal coefficients $(\theta(A))$, $A \subset \{1,\dots,d\}$, $A \neq \emptyset$ and ${\bm X}^*$ the Choquet (Tawn-Molchanov) random vector with identical extremal coefficients. Then
     \[{\bm X}^* \, \leq_{\mathrm{PQD}}\,  {\bm X}.
\]
\end{corollary}

\subsection{Parametric models}
\label{sec:parametric-results}

In general, parametric families of multivariate distributions do not necessarily exhibit stochastic orderings. 
One of the few more interesting known examples among multivariate max-stable distributions is the Dirichlet family, for which it has been shown that it is ordered in the symmetric case \citep[Proposition 4.4]{afz_15}, that is, for $\alpha \leq \beta$ we have 
\begin{align}
\label{eq:MaxDirSymmetricOrdering}
\mathrm{MaxDir}((\alpha,\alpha,\dots,\alpha)) \geq_{\mathrm{lo}} \mathrm{MaxDir}((\beta,\beta,\dots,\beta)).
\end{align}
 Figure~\ref{fig:Dirichletdepsets} illustrates \eqref{eq:MaxDirSymmetricOrdering} in the bivariate situation and shows a bivariate example that these distributions are otherwise not necessarily ordered in the asymmetric case.

\begin{figure}[htb]
    \centering
    \includegraphics[width=0.3\textwidth]{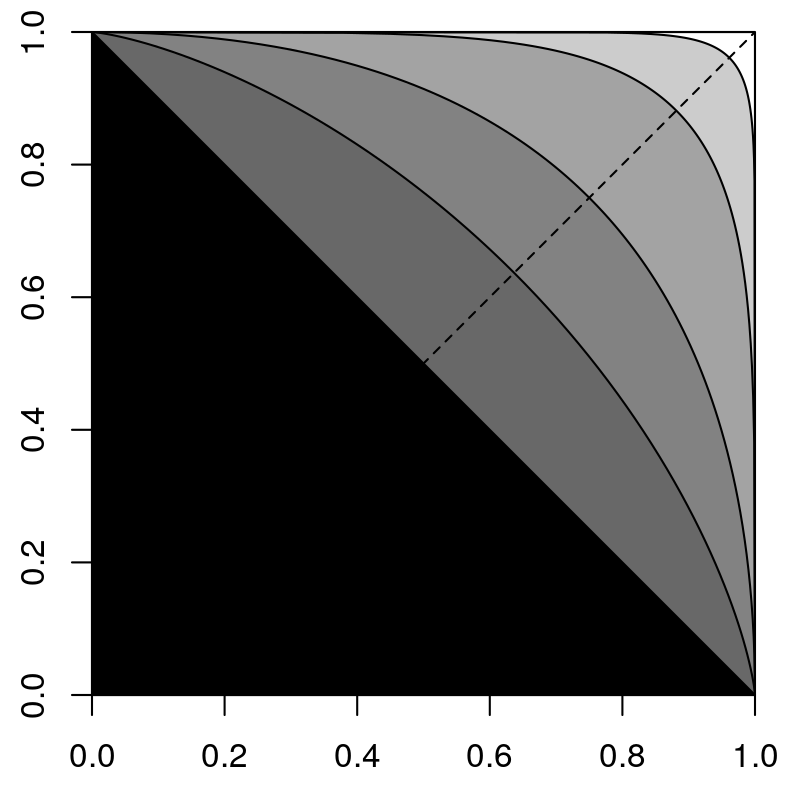}
    \includegraphics[width=0.3\textwidth]{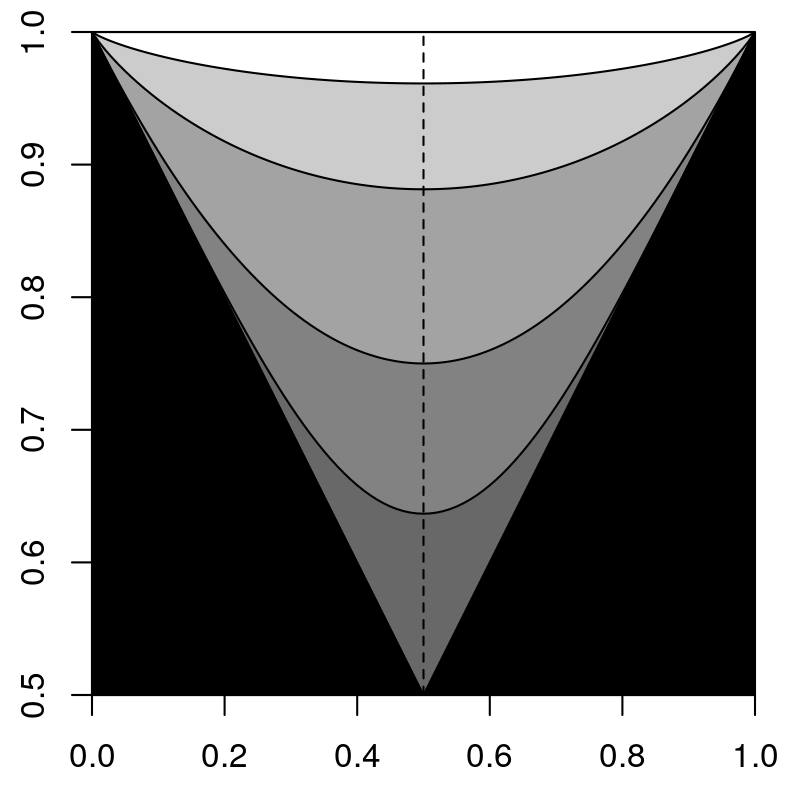}
    \\
    \includegraphics[width=0.3\textwidth]{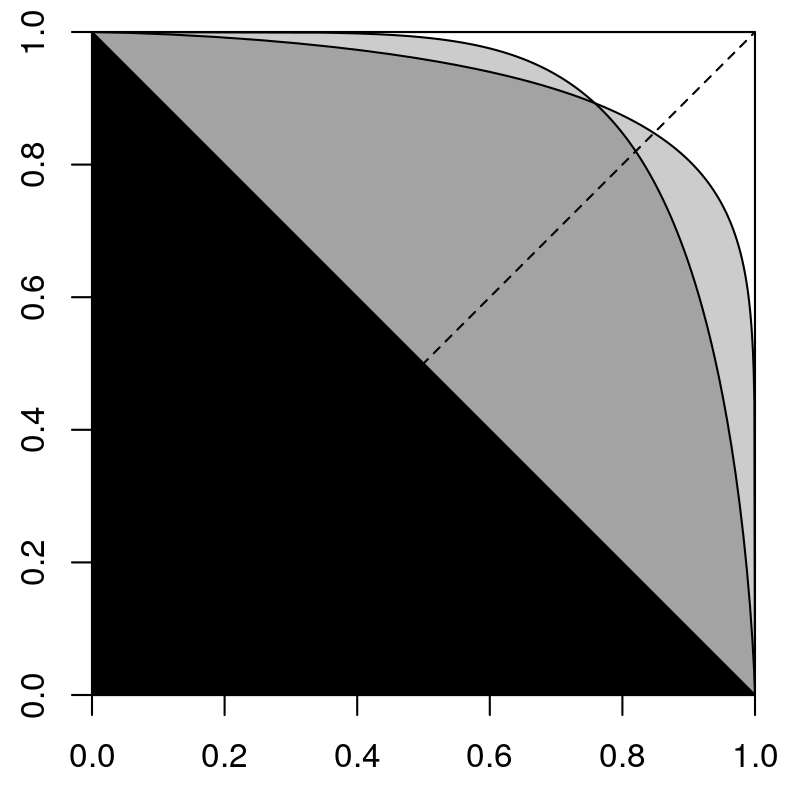}
    \includegraphics[width=0.3\textwidth]{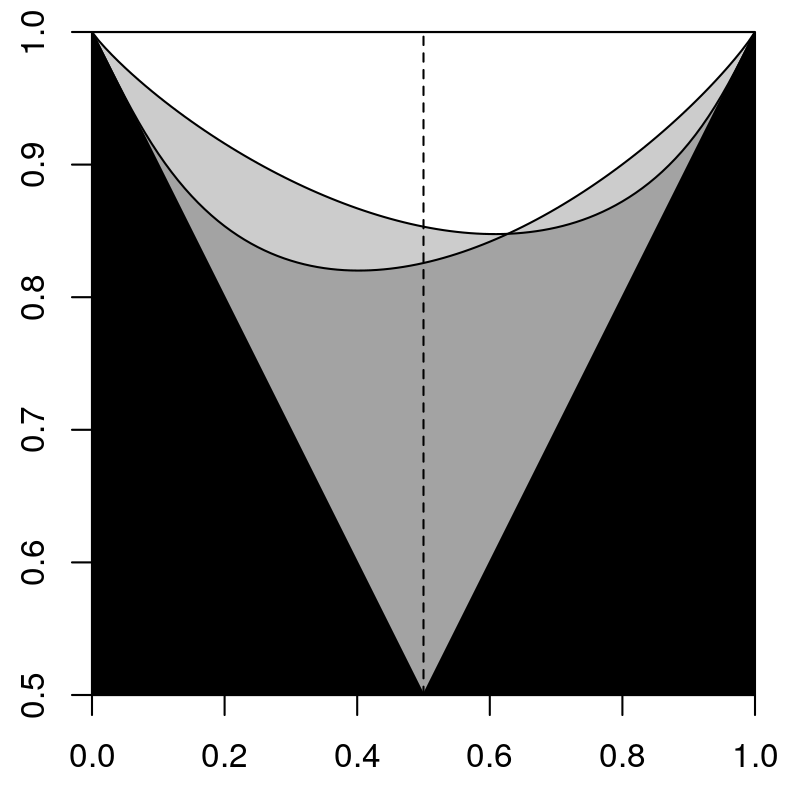}
    \caption{\small Top: Nested max-zonoids (left) and ordered (hypographs of) Pickands dependence functions (right) from the fully symmetric Dirichlet family for $\alpha \in \lbrace 0.0625, 0.25, 1, 4\rbrace$. 
    Smaller values of $\alpha$ correspond to larger sets and larger Pickands dependence functions and are closer to the independence model represented by the box $[0,1]^2$ or the constant function, which is identically 1. The fully dependent model is represented in black.
    Bottom: Non-nested max-zonoids and non-ordered Pickands dependence function from the asymmetric Dirichlet family for $ (\alpha_1,\alpha_2) \in \{(0.15,12),(4,0.2)\}$.}
    \label{fig:Dirichletdepsets}
\end{figure}

Here, we extend \eqref{eq:MaxDirSymmetricOrdering} in several ways: (i) going beyond the symmetric situation considering the fully asymmetric model, (ii) considering PQD/concordance order, (iii) shortening the proof by exploiting a connection to the theory of majorisation, cf.~Appendix~\ref{app:DirichletHR}. Figure~\ref{fig:AsymDirichletdepsets} provides an illustration of the stochastic ordering for the asymmetric Dirichlet family in the bivariate case. In Figure~\ref{fig:DirichletAngularDensities} we see how the mass of the angular measure of the symmetric and asymmetric Dirichlet model is more concentrated from left plot to right plot. This also corresponds to their stochastic ordering, with the right one being the most dominant model in terms of PQD order.

\begin{theorem}[PQD/concordance order of Dirichlet family]
\label{thm:Dirichlet-PQD}
Consider the max-stable Dirichlet family from Theorem/Definition~\ref{thmdef:Dirichlet}.
If ${\alpha}_i \leq {\beta}_i$, $i=1,\dots,d$, then \[\mathrm{MaxDir}({\alpha_1,\alpha_2,\dots,\alpha_d})  \leq_{\mathrm{PQD}} \mathrm{MaxDir}({\beta_1,\beta_2,\dots,\beta_d}).\]
\end{theorem}

\begin{figure}[htb]
    \centering
    \includegraphics[width=0.3\textwidth]{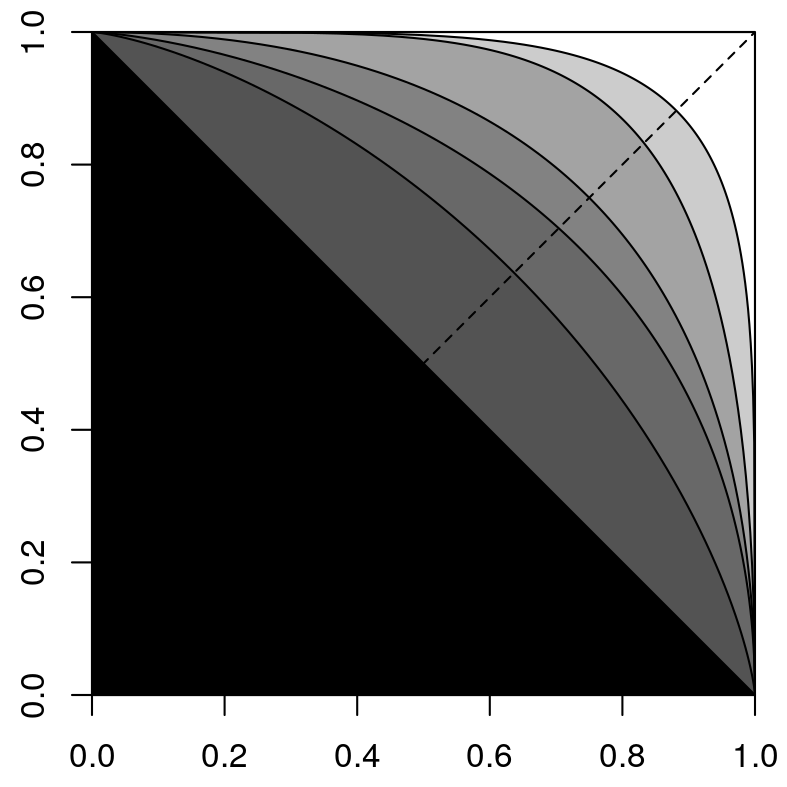}
    \includegraphics[width=0.3\textwidth]{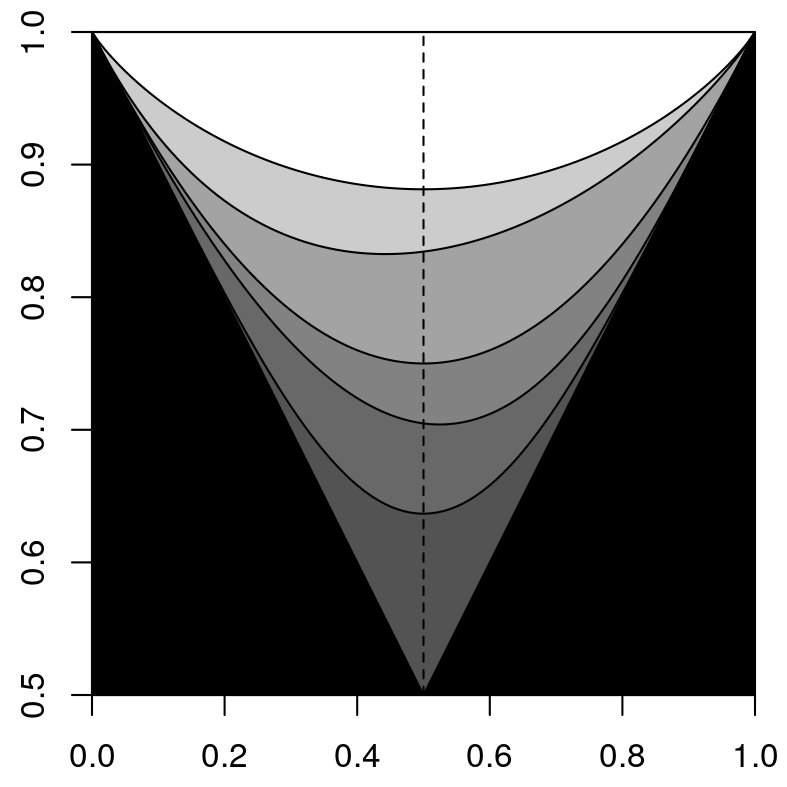}
    \caption{\small Nested max-zonoids and ordered Pickands dependence functions of the asymmetric max-stable Dirichlet family for $(\alpha_1,\alpha_2)\in\lbrace(0.25,0.25),(1,0.25),(1,1),(1,4),(4,4)\rbrace$. Componentwise smaller values of $(\alpha_1,\alpha_2)$ correspond to larger sets and larger Pickands dependence functions and are closer to the independence model.}
    \label{fig:AsymDirichletdepsets}
\end{figure}

\begin{example}
In order to draw attention to some further consequences of Theorem~\ref{thm:Dirichlet-PQD}, let ${\bm X} \sim \mathrm{MaxDir}({\bm \alpha})$ and ${\bm Y} \sim \mathrm{MaxDir}({\bm \beta})$ where $\alpha_i \leq \beta_i$, $i=1,\dots,d$, so that ${\bm X} \leq_{\mathrm{PQD}} {\bm Y}$, hence ${\bm X} \leq_{\mathrm{uo}} {\bm Y}$ and ${\bm X} \geq_{\mathrm{lo}} {\bm Y}$, which implies
\begin{align*}
    \min_{i=1,\dots,d}(a_iX_i)\leq_{\mathrm{st}}\min_{i=1,\dots,d}(a_iY_i) \qquad \text{for all ${\bm a}\in (0,\infty]^d$.}\\
      \max_{i=1,\dots,d}(a_iX_i)\geq_{\mathrm{st}}\max_{i=1,\dots,d}(a_iY_i) \qquad \text{for all ${\bm a}\in [0,\infty)^d$,}
\end{align*}
cf.~\eqref{eq:UO-nonnegative-min}, \eqref{eq:LO-nonnegative-max} and Lemma~\ref{lem:MaxDirClosure}.
       Exemplarily, we consider
a range of trivariate symmetric and asymmetric max-stable Dirichlet distributions $\mathrm{MaxDir}(\alpha_1,\alpha_2,\alpha_3)$ with parameters $(\alpha_1,\alpha_2,\alpha_3)$ given in Figure~\ref{fig:DirichletAngularDensities}. 
The colouring is chosen such that red models PQD-dominate blue models, which PQD-dominate black models. 

In addition, we consider the portfolio with equal weights $(1,1,1)$ and the resulting min-projections $\min(X_1,X_2,X_3)$ and max-projections $\max(X_1,X_2,X_3)$, where $(X_1,X_2,X_3)\sim \mathrm{MaxDir}(\alpha_1,\alpha_2,\alpha_3)$. Figures~\ref{fig:MinCDFs} and \ref{fig:MaxCDFs} display their distribution functions on the Gumbel scale. 
As commonly of interest for extreme value distributions, instead of the quantile function $Q$, we show the equivalent return level plot, which displays the return levels $Q(1-p)$ for the return period of $1/p$ observations. 
The plots of these functions are based on empirical estimates from one million simulated observations from the respective models, and their orderings are as expected from the theory, i.e.\ quantile functions increase as the dominance of the model grows, while distribution functions decrease.
\end{example}

\def\widthCDFs{0.3\linewidth}

\begin{figure}[htb]
    \centering\small
    \begin{tabular}{ccc}
    & \textsf{Distribution functions}
    & \textsf{Return levels}
    \\
    \rotatebox{90}{\hspace*{1cm} \textsf{Symmetric case}} 
    &\includegraphics[width=\widthCDFs]{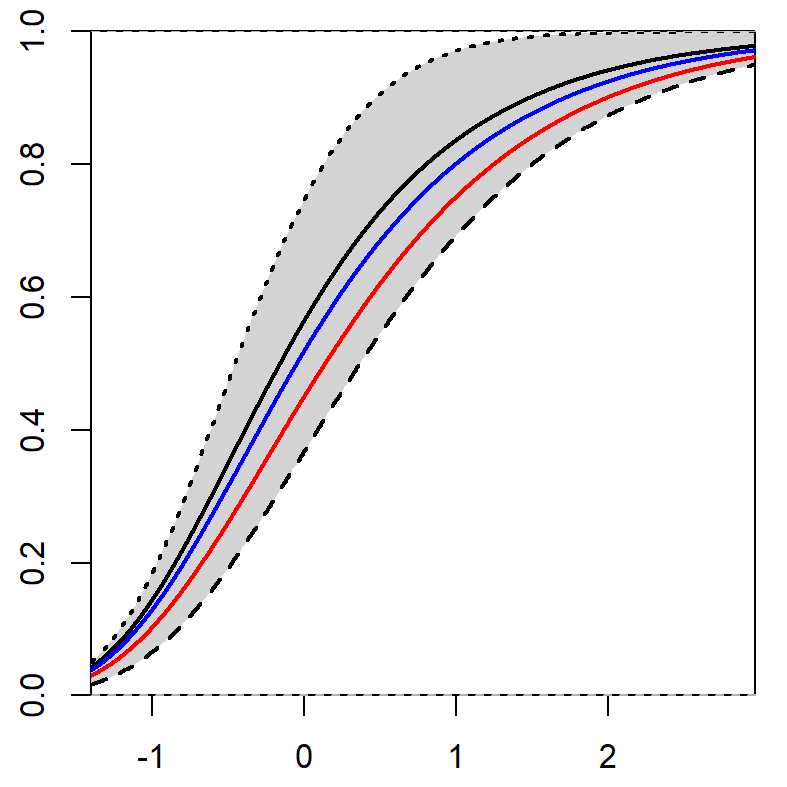}
    &\includegraphics[width=\widthCDFs]{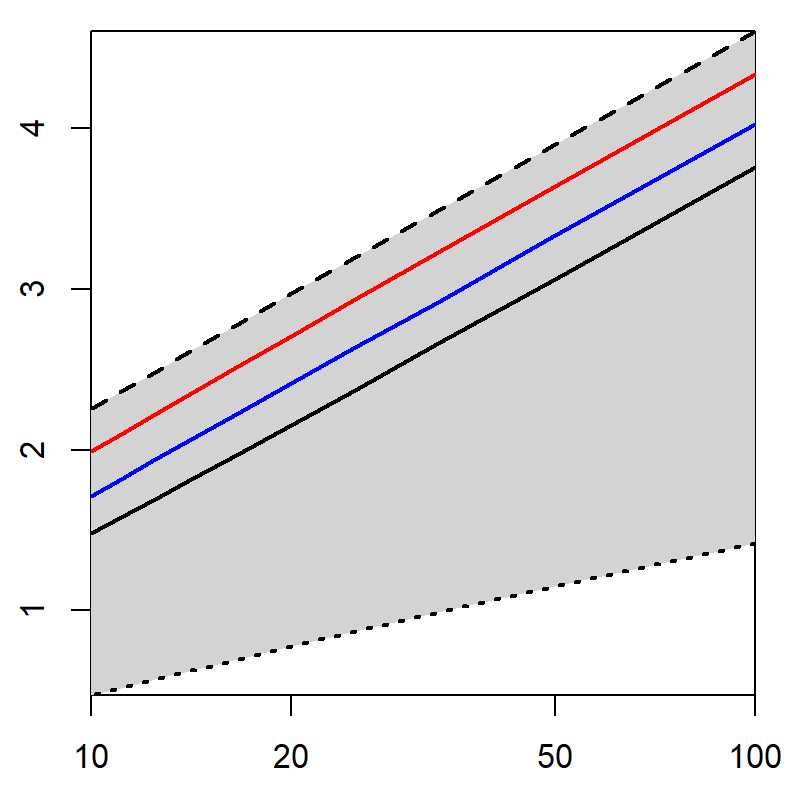}
    \\
    \rotatebox{90}{\hspace*{1cm} \textsf{Asymmetric case}}
    &\includegraphics[width=\widthCDFs]{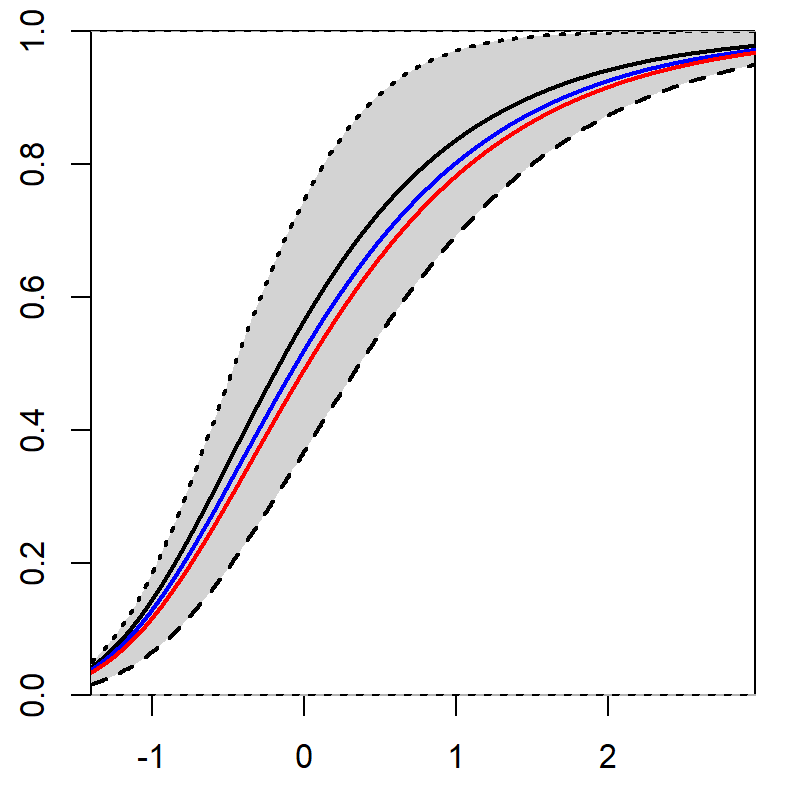}
    &\includegraphics[width=\widthCDFs]{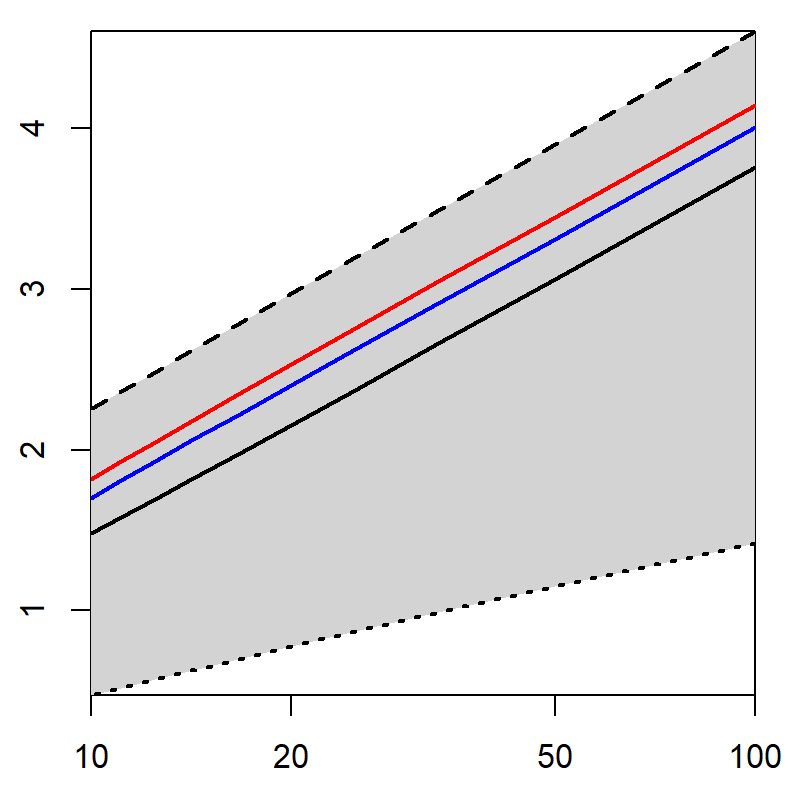}
    \end{tabular}
    \caption{\small Distribution functions (left) and return levels (right) for return periods 10 to 100 (on logarithmic scale) of $\min(X_1,X_2,X_3)$, where $(X_1,X_2,X_3)\sim \mathrm{MaxDir}(\alpha_1,\alpha_2,\alpha_3)$ on standard Gumbel scale with ${\bm \alpha}=(\alpha_1,\alpha_2,\alpha_3)$ as chosen in Figure~\ref{fig:DirichletAngularDensities}.  Top: symmetric case; bottom: asymmetric case. Black, blue and red colouring encodes the matching with Figure~\ref{fig:DirichletAngularDensities}. The grey areas represent the range between the fully dependent (dashed line) and fully independent (dotted line) cases.}
    \label{fig:MinCDFs}
\end{figure}

\begin{figure}[htb]
    \centering\small
    \begin{tabular}{ccc}
    & \textsf{Distribution functions}
    & \textsf{Return levels}
    \\
    \rotatebox{90}{\hspace*{1cm} \textsf{Symmetric case}} 
    &\includegraphics[width=\widthCDFs]{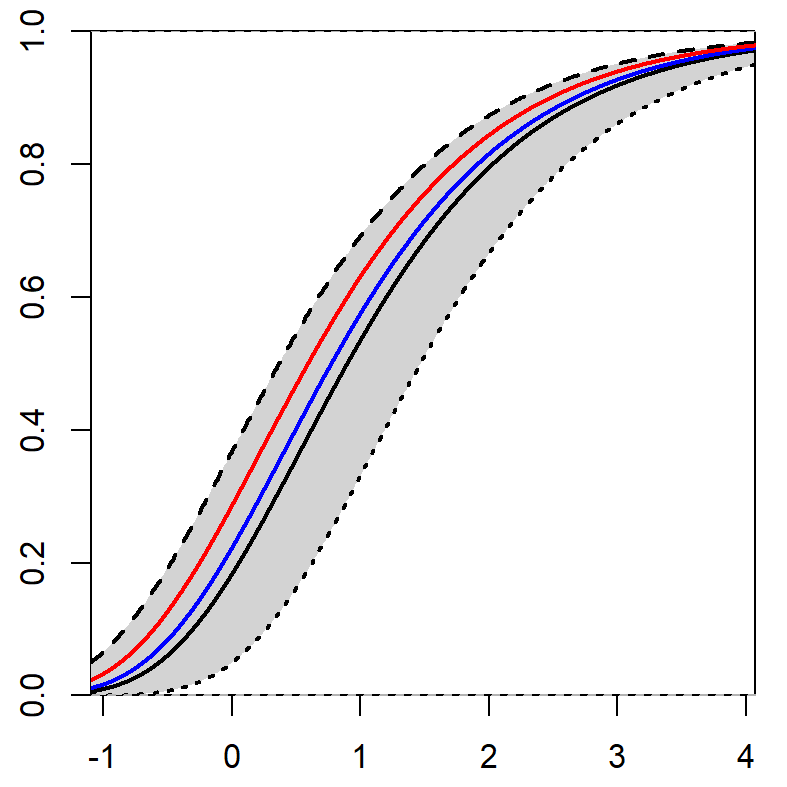}
    &\includegraphics[width=\widthCDFs]{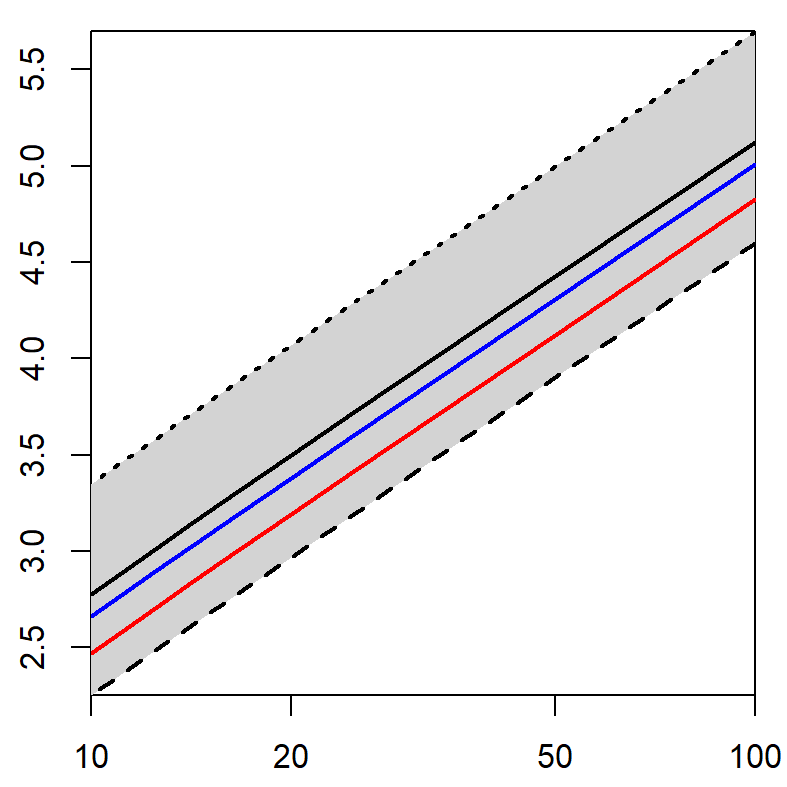}
    \\
    \rotatebox{90}{\hspace*{1cm} \textsf{Asymmetric case}}
    &\includegraphics[width=\widthCDFs]{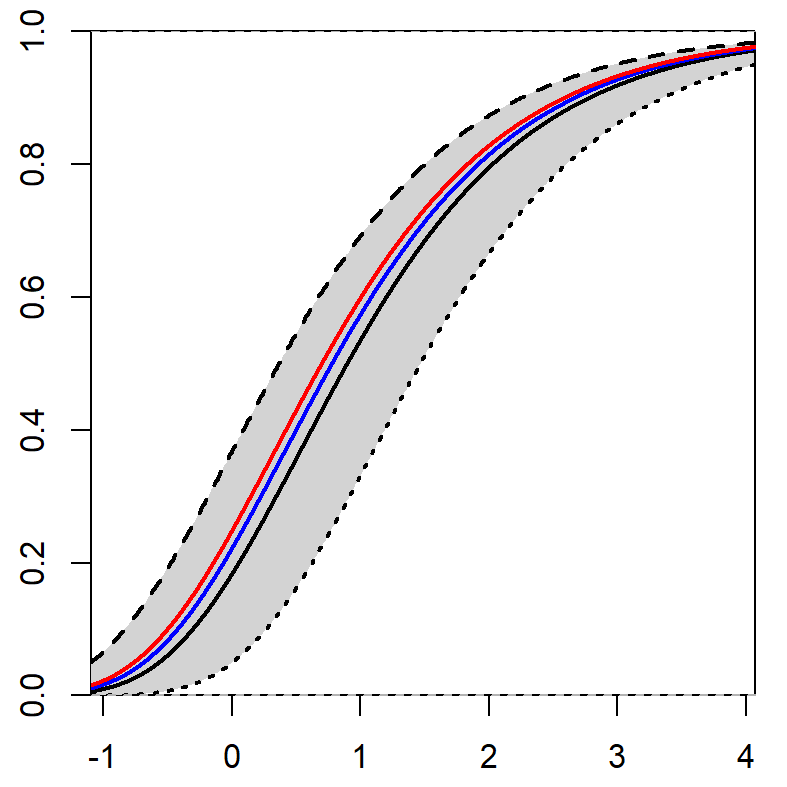}
    &\includegraphics[width=\widthCDFs]{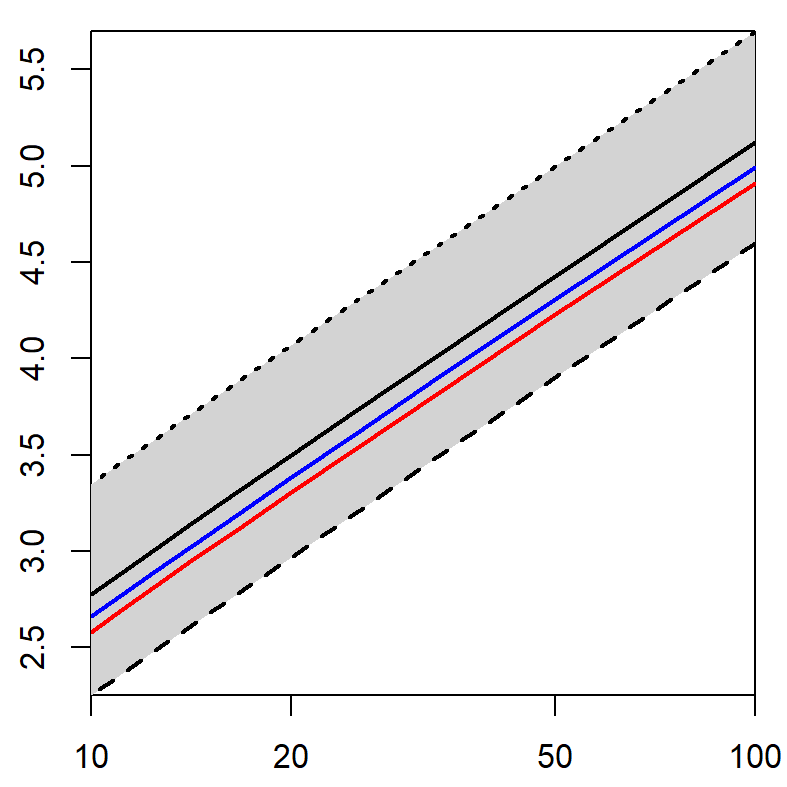}
    \end{tabular}
    \caption{\small Distribution functions (left) and return levels (right) for return periods 10 to 100 (on logarithmic scale) of $\max(X_1,X_2,X_3)$, where $(X_1,X_2,X_3)\sim \mathrm{MaxDir}(\alpha_1,\alpha_2,\alpha_3)$ on standard Gumbel scale with ${\bm \alpha}=(\alpha_1,\alpha_2,\alpha_3)$ as chosen in Figure~\ref{fig:DirichletAngularDensities}.  Top: symmetric case; bottom: asymmetric case. Black, blue and red colouring encodes the matching with Figure~\ref{fig:DirichletAngularDensities}. The grey areas represent the range between the fully dependent (dashed line) and fully independent (dotted line) cases.}
    \label{fig:MaxCDFs}
\end{figure}

Another prominent family of multivariate max-stable distributions that turns out to  be stochastically ordered in the PQD/concordance order is the H\"usler-Rei{\ss} family. It can be shown using the limit result from Theorem~\ref{thm:triangularArrayHR} together with Slepian's normal comparison lemma and some closure properties of the PQD/concordance order. Figure~\ref{fig:HRdepsets} provides an illustration in terms of nested max-zonoids and ordered Pickands dependence functions in the bivariate case. However, while these models are ordered, we would like to point out that none of the typically chosen families of log-Gaussian generators satisfy any of the orthant orders, cf.~Example~\ref{example:HR-generators-not-ordered}.

\begin{theorem}[PQD/concordance order of H\"usler-Rei{\ss} family]
\label{thm:HRm}
Consider the max-stable H\"usler-Rei{\ss} family from Theorem/Definition~\ref{thmdef:HR}.
If ${\gamma}_{i,j} \leq {\widetilde \gamma}_{i,j}$ for all $i,j \in \{1,\dots,d\}$, then 
\[\mathrm{HR}({\bm \gamma})  \geq_{\mathrm{PQD}} \mathrm{HR}(\widetilde {\bm \gamma}).\]
\end{theorem}

\begin{remark}
With almost identical proof, cf.~Appendix~\ref{app:DirichletHR}, we may even deduce 
\[\mathrm{HR}({\bm \gamma})  \geq_{\mathrm{sm}} \mathrm{HR}(\widetilde {\bm \gamma}),\]
where $\geq_{\mathrm{sm}}$ denotes the \emph{supermodular} order, cf.~\citet{mueller_02} Section~3.9 or \citet{shsh_07} Section~9.A.4. We have therefore included the respective arguments in the proof, too, although considering the supermodular order is otherwise beyond the scope of this article.
\end{remark}

\begin{figure}[htb]
    \centering
    \includegraphics[width=0.3\textwidth]{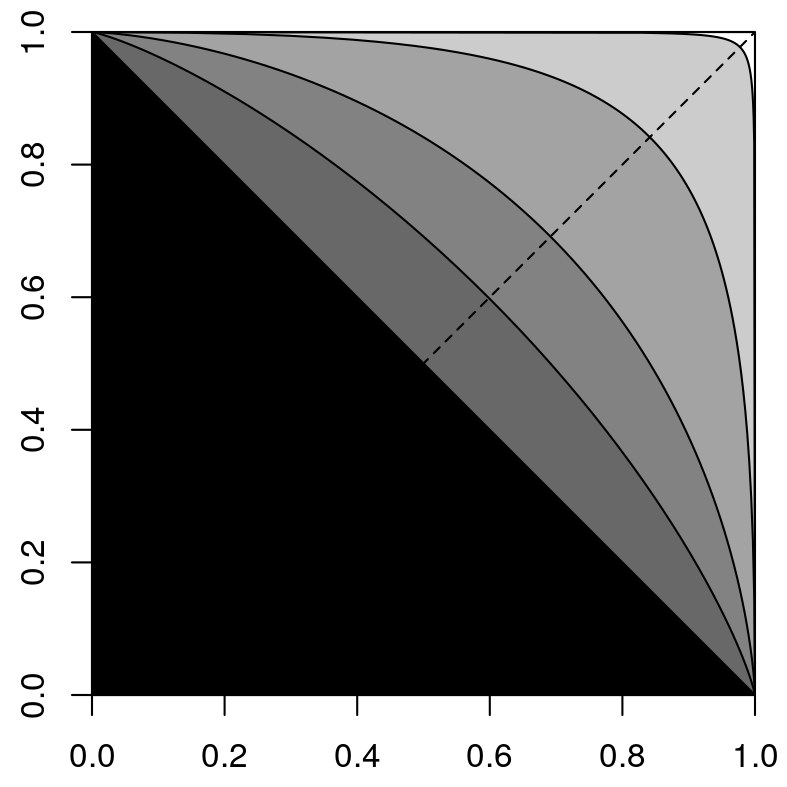}
       \includegraphics[width=0.3\textwidth]{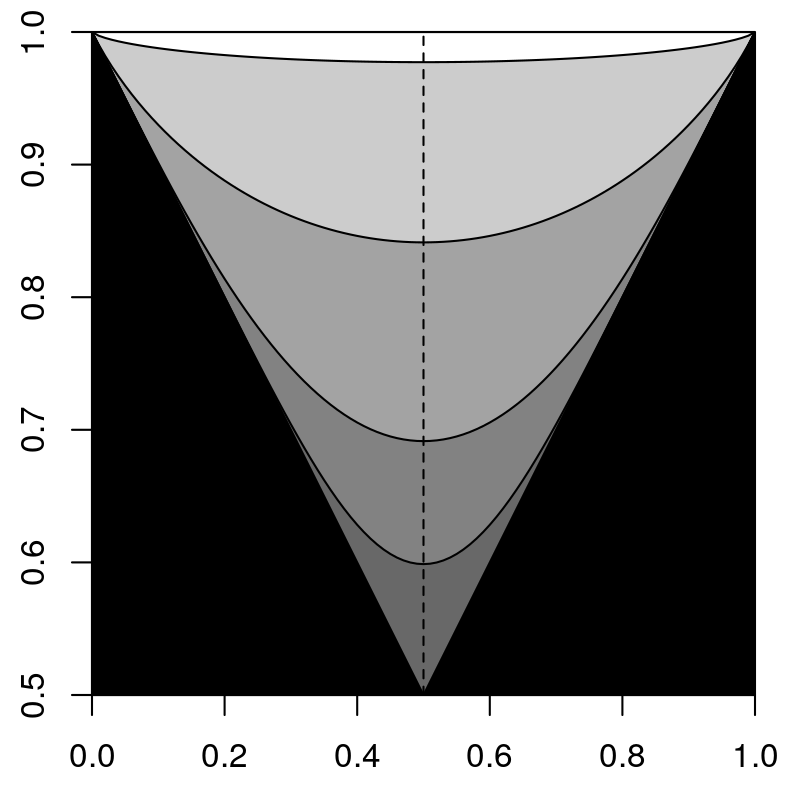}
    \caption{\small Nested max-zonoids and ordered Pickands dependence functions from the bivariate H\"usler-Rei{\ss} family for $\sqrt{\gamma} \in \{0.5,1,2,4\}$. Larger values of $\gamma$ correspond to larger sets and larger Pickands dependence functions and are closer to the independence model.}
    \label{fig:HRdepsets}
\end{figure}

\begin{remark}
    Theorem~\ref{thm:HRm} includes the assumption that both parameter matrices ${\bm \gamma}$ and $\widetilde {\bm \gamma}$ constitute a valid set of parameters for the H\"usler-Rei{\ss} model, i.e.\ they need to be elements of $\mathcal{G}_d$. In dimensions $d \geq 3$ it is possible that increasing (or decreasing) any of the parameters of a given valid ${\bm \gamma}$ will result in a set of parameters that is not valid for the H\"usler-Rei{\ss} model.
\end{remark}

\begin{remark}
      Since the orthant orders are closed under independent conjunction, 
      Theorem~\ref{thm:HRm} extends to the generalised H\"usler-Rei{\ss} model, where we can allow some parameter values of ${\bm \gamma}$ to assume the value $\infty$, as long as  ${\bm \gamma}$  remains negative definite in the extended sense (see Remark~\ref{ref:HRextended}).
\end{remark}

\begin{example}[Ordering of $G$/$\widetilde G$ does not imply generator ordering for ${\bm Z}$/$\widetilde {\bm Z}$ -- the case of H\"usler-Rei{\ss} log-Gaussian generators]\label{example:HR-generators-not-ordered}
  Consider the non-degenarate bivariate H\"usler-Rei{\ss} model with $\gamma_{12}=\gamma_{21}=\gamma \in (0,\infty)$ and let additionally $a\in [0,1]$. Then the zero mean bivariate Gaussian model $(W_1,W_2)^\top$ with $\EE(W_1)=\gamma a^2$, $\EE(W_2)=\gamma (1-a)^2$, $\mathrm{Cov}(W_1,W_2)=0.5\gamma \cdot (a^2+(1-a)^2-1)$ satisfies $\EE(W_1-W_2)^2=\gamma$, hence leads to a generator for the bivariate H\"usler-Rei{\ss} distribution in the sense of Theorem/Definition~\ref{thmdef:HR}. WLOG $a \in (0,1]$ (otherwise consider $1-a$ instead of $a$). Then $\log(Z_1)$ follows a non-degenerate univariate Gaussian distribution with mean $-0.5 \gamma a^2$ and variance $\gamma a^2$. The family of such distributions is not ordered in $\gamma>0$ (cf.~e.g.~\citet{shsh_07} Example~1.A.26 or \citet{mueller_02} Theorem~3.3.13). Hence, the bivariate family $(\log(Z_1),\log(Z_2))^\top$ can also not be ordered according to orthant order, nor can any multivariate family, for which this constitutes a marginal family. Accordingly, the corresponding log-Gaussian generators ${\bm Z}$ of the H\"usler-Rei{\ss} model will not be ordered, even when the resulting max-stable model and exponent measures are ordered as seen in Theorem~\ref{thm:HRm}.
\end{example}

While Dirichlet and H\"usler-Rei{\ss} families are ordered in the PQD/concordance sense according to the natural ordering of their parameter spaces, we would like to provide some examples that show that UO and LO ordering among simple max-stable distributions are in fact not equivalent. 

To this end, we revisit the Choquet max-stable model from Section~\ref{sec:choquet}. We write $\mathrm{Choquet}_{\mathrm{EC}}(\theta)$ for the simple max-stable Choquet distribution if it is parameterised by its extremal coefficients $\theta(A)$, $A \subset \{1,\dots,d\}$, $A \neq \emptyset$ and $\mathrm{Choquet}_{\mathrm{TD}}(\chi)$  if it is parameterised by its tail dependence coefficients $\chi(A)$, $A \subset \{1,\dots,d\}$, $A \neq \emptyset$.

\begin{lemma}[LO and UO order of Choquet family/Tawn-Molchanov model]
\label{lemma:Choquet-LO-UO}
Consider the family of max-stable Choquet models from Section~\ref{sec:choquet}. Then the LO order is characterised by the ordering of extremal coefficients, we have
\begin{align*}
\theta \leq \widetilde \theta \qquad &\iff \qquad
\mathrm{Choquet}_{\mathrm{EC}}(\theta)
\leq_{\mathrm{lo}} \mathrm{Choquet}_{\mathrm{EC}}(\widetilde \theta);
\end{align*}
and the UO order is characterised by the ordering of tail dependence coefficients, that is
\begin{align*}
\chi \leq \widetilde \chi \qquad &\iff \qquad
\mathrm{Choquet}_{\mathrm{TD}}(\chi)
\leq_{\mathrm{uo}} \mathrm{Choquet}_{\mathrm{TD}}(\widetilde \chi).
\end{align*}
\end{lemma}

As we know already from Theorem~\ref{thm:OO-characterization} part c),
in dimension $d=2$, it is easily seen that $\chi \leq \widetilde \chi$ is equivalent to $\theta \geq \widetilde \theta$, alternatively recall $\theta_{12}+\chi_{12}=2$. Starting from dimension $d=3$, this is no longer the case and one can easily construct examples, where only LO or UO ordering holds.

\begin{example}\label{example:UO-LO-not-equivalent}
Table~\ref{tab:Choquet-3-exchangeable} lists valid parameter sets for four different trivariate Choquet models. Among these, we can  easily read off that
\begin{itemize}
\itemsep0mm
    \item $B \leq_{\mathrm{uo}}D$, but there is no order between $B$ and $D$ according to lower orthants;
    \item $C \leq_{\mathrm{lo}} B$, but there is no order between $B$ and $C$ according to upper orthants.
\end{itemize}
Of course, it is still possible that Choquet models are ordered according to PQD order, e.g.
\begin{itemize}
\itemsep0mm
    \item $A \leq_{\mathrm{PQD}}B$.
\end{itemize}
It is also possible to have both LO and UO order in the same direction, e.g. 
\begin{itemize}
\itemsep0mm
    \item $A \leq_{\mathrm{uo}}C$ and $A \leq_{\mathrm{lo}}C$.
\end{itemize}
However, note that such an order can only arise if the bivariate marginal distributions all agree.
\end{example}

\begin{table}[hbt]
    \centering
       \caption{Valid parameter sets of trivariate Choquet models, cf.~Section~\ref{sec:choquet}, that are exchangeable so that parameters $\tau_A$, $\chi_A$, $\theta_A$ depend on sets $A$ only through their cardinality. Since $\chi_1=\theta_1=1$ these parameters need not be listed. We have $\tau_1+2\tau_{12}+\tau_{123}=1$, $\chi_{123}=\tau_{123}$, $\chi_{12}=\chi_{123}+\tau_{12}$,  $\theta_{12}=1+\tau_1+\tau_{12}$ and $\theta_{123}=\theta_{12}+\tau_1$.}
    \label{tab:Choquet-3-exchangeable}
    \vspace*{2mm}
    \begin{tabular}{|c|ccc|cc|cc|}
    \hline
         &  $\tau_1$ & $\tau_{12}$ & $\tau_{123}$ 
         &  $\chi_{12}$ & $\chi_{123}$ 
           & $\theta_{12}$ & $\theta_{123}$
         \\ \hline
         A & 0.3 & 0.2 & 0.3 & 0.5 & 0.3 & 1.5 & 1.8 \\
         B & 0.1 & 0.3 & 0.3 & 0.6 & 0.3 & 1.4 & 1.5 \\
         C & 0.4 & 0.1 & 0.4 & 0.5 & 0.4 & 1.5 & 1.9\\
         D & 0.3 & 0.0 & 0.7 & 0.7 & 0.7 & 1.3 & 1.6\\ \hline
    \end{tabular}
\end{table}

\paragraph{Acknowledgements}
We would like to thank two anonymous referees who have positively challenged us in two directions: (i) not only considering the order according to lower orthants, but to study the PQD-order, and (ii) extending the bivariate to the multivariate H\"usler-Rei{\ss} model. Their encouragement to pursue these routes has helped us to significantly broaden the scope of this work.  
 MC thankfully acknowledges funding from the Maths DTP 2020 -- Engineering and Physical Sciences Research Council grant EP/V520159/1 -- at the School of Mathematics of Cardiff University.

{\small
\bibsep=3pt
\bibliographystyle{agsm}

} 

\appendix
\section{Proofs}
\label{app:proofs}

\subsection{Proofs concerning fundamental order relations}\label{app:proofs-fundamentals}

\begin{proof}[Proof of Theorem~\ref{thm:OO-characterization}]
In what follows, let ${\bm X} \sim G$ and $\widetilde {\bm X} \sim \widetilde G$.

a) Because of \eqref{eq:LO-nonnegative-max}, it suffices to compare $G({\bm x})$ and $\widetilde G({\bm x})$ for  ${\bm x} \in (0,\infty)^d$ only. The same is true for $\ell$ and $\widetilde \ell$ as they are continuous on $[0,\infty)^d$. At the same time the test sets for the relation $\Lambda \leq_{\mathrm{lo}}\widetilde \Lambda$ in Definition~\ref{def:orders-Lambda} are precisely of the form $\RR^d \setminus L_{{\bm x}}$, where ${\bm x} \in (0,\infty)^d$. So the equivalence of (i), (ii), (iii) and (iv) follows directly from the relations
\[
G({\bm x})
=\PP({\bm X} \in L_{\bm x})
=\exp(-\Lambda([0,\infty]^d \setminus L_{\bm x}))
=\exp(-\Lambda(\RR^d \setminus L_{\bm x}))
\]
with
\[
\Lambda(\RR^d \setminus L_{\bm x}) = \ell(1/x_1,\dots,1/x_d)= \EE(\max(Z_1/x_1,\dots,Z_d/x_d)),
\]
and the respective tilde-counterparts. Likewise, the equivalence of (iv) and (v) is immediate from \eqref{eq:K-from-ell} and \eqref{eq:ell-from-K}.

b) We start by showing the equivalence between (ii) and (iii). The test sets for the relation $\Lambda \leq_{\mathrm{uo}}\widetilde \Lambda$ in Definition~\ref{def:orders-Lambda} are precisely the upper orthants $U_{{\bm x}}$, where at least one component of ${\bm x}$ is larger than zero.  Let ${\bm a} \in (0,\infty)^d$ and $A \subset \{1,\dots,d\}$, $A \neq \emptyset$. Define ${\bm x} \in \RR^d$ by setting $x_i=1/a_i$ if $i \in A$ and $x_i=-1$ else. Then $U_{\bm x}$ is an admissible test set and 
\begin{align}\label{eq:U-minZ}
  \Lambda(U_{\bm x}) = \Lambda\bigg(\Big\{{\bm y} \in [0,\infty)^d\setminus\{{\bm 0}\} \,:\, \min_{i\in A}(a_iy_i)>1\Big\}\bigg)
=\EE\Big(\min_{i \in A}(a_i Z_i)\Big).  
\end{align}
Likewise, $\widetilde\Lambda(U_{\bm x})=\EE(\min_{i \in A}(a_i \widetilde Z_i))$ and we may deduce the implication (ii)$\Rightarrow$(iii). Conversely, assume (iii) and note that the same argument implies $\Lambda(U_{\bm x})\leq \widetilde \Lambda(U_{\bm x})$ for any ${\bm x}$, which has at least one positive component, whilst all other components of ${\bm x}$ are negative. What remains to be seen is the same relation for upper orthants $U_{\bm x}$, for which at least one component of ${\bm x}$ is positive, but where among the non-positive components, there may be zeroes. Let ${\bm x} \in \RR^d$ be such a vector. For $n \in \NN$ let ${\bm x}_n \in \RR^d$ be an identical vector, but with zero entries replaced by $1/n$. Then $\Lambda(U_{{\bm x}_n})\leq \widetilde \Lambda(U_{{\bm x}_n})$ for all $n \in \NN$ by the previous argument, whilst $U_{{\bm x}_n} \uparrow U_{\bm x}$, such that 
$\Lambda(U_{{\bm x}_n}) \to \Lambda(U_{{\bm x}})$ and $\widetilde \Lambda(U_{{\bm x}_n}) \to \widetilde \Lambda(U_{{\bm x}})$ as $n \to \infty$. This shows (iii)$\Rightarrow$(ii).

 Next, we establish (i)$\Rightarrow$(iii). Assume (i). Since the order UO is closed under marginalisation, it suffices to consider $A=\{1,\dots,d\}$ in (iii), see also Lemma~\ref{lem:marZ}. Set $x_i=1/a_i$, $i=1,\dots,d$, such that \eqref{eq:U-minZ} holds (as well as the tilde-version) and note that the closure of $U_{\bm x}$ in $[0,\infty]^d \setminus \{{\bm 0}\}$ is a continuity set for each of the $(-1)$-homogeneous measures $\Lambda$ and $\widetilde \Lambda$. Hence, since each max-stable vector satisfies its own Domain-of-attraction conditions (cf.~e.g.~\citet{resnick_08} Section~5.4.2), we have 
 \[
 \Lambda(U_{\bm x})=\lim_{n \to \infty} n \PP({\bm X} \in n U_{\bm x})=\lim_{n \to \infty} n \PP({\bm X} \in  U_{n{\bm x}})
 \]
 and the analog for $\widetilde  \Lambda$ and $\widetilde {\bm X}$. The implication (i)$\Rightarrow$(iii) follows.

 Lastly, let us establish (iii)$\Rightarrow$(i). Suppose (iii) holds. We abbreviate $\chi^{(a)}(A)=\EE(\min_{i \in A}(a_i Z_i))$ and analogously $\widetilde \chi^{(a)}(A)=\EE(\min_{i \in A}(a_i \widetilde Z_i))$, such that (iii) translates into
  \[\chi^{(a)}(A) \leq \widetilde \chi^{(a)}(A)\]
 for all ${\bm a} \in (0,\infty)^d$ and $A \subset \{1,\dots,d\}$, $A \neq \emptyset$.
With $\theta^{(a)}(A)=\EE(\max_{i \in A}(a_i Z_i))$  we find that 
\[ \chi^{(a)}(A)=\sum_{I \subset A, \, I \neq \emptyset} (-1)^{|I|+1} \theta^{(a)}(I) \qquad \text{and} \qquad  \theta^{(a)}(A)=\sum_{I \subset A, \, I \neq \emptyset} (-1)^{|I|+1} \chi^{(a)}(I)\]
(similarly to \eqref{eq:chi-theta} and analogously for the tilde-version), where $\theta^{(a)}$ can be interpreted as directional extremal coefficient function. It is easily seen that  $\theta^{(a)}$ with $\theta^{(a)}(\emptyset)=0$ is union-completely alternating, cf.~Lemma~\ref{lem:ell-maxCA-theta-unionCA}.

 Because of \eqref{eq:UO-nonnegative-min}, in order to arrive at (i), it suffices to establish 
 \[
\PP\Big(\min_{i=1,\dots,d}(a_i X_i)>1\Big) \leq \PP\Big(\min_{i=1,\dots,d}(a_i \widetilde X_i)>1\Big)
 \]
 for all ${\bm a} \in (0,\infty)^d$. In the notation of Lemma~\ref{lem:ell-maxCA-theta-unionCA} and with $g(x)=1-\exp(-x)$,
 the left-hand side can be rewritten as
 \begin{align*}
     \PP\Big(\min_{i=1,\dots,d}(a_i X_i)>1\Big)
 &= 1-\sum_{I \subset \{1,\dots,d\}, \, I \neq \emptyset} (-1)^{|I|+1} \PP\Big(\max_{i \in I}(a_i X_i) \leq 1\Big)\\
 &= - \sum_{I \subset \{1,\dots,d\}} (-1)^{|I|+1} \exp\big(-\ell({\bm a}_I)\big)\\
  &=  \sum_{I \subset \{1,\dots,d\}} (-1)^{|I|+1}   - \sum_{I \subset \{1,\dots,d\}} (-1)^{|I|+1} \exp\big(-\theta^{(a)}(I)\big)\\
  &=  \sum_{I \subset \{1,\dots,d\}} (-1)^{|I|+1} g\big(\theta^{(a)}(I)\big)
 \end{align*}
 (and analogously for the tilde-version). The assertion follows then directly from Proposition~\ref{prop:key-OO-fundamentals}, since $g$ is a Bernstein function. 

 c) The statement follows from the relation
 \[\EE(\min(a_1Z_1,a_2Z_2)) + \EE(\max(a_1Z_1,a_2Z_2))  = \EE(\min(a_1\widetilde Z_1,a_2 \widetilde Z_2)) + \EE(\max(a_1 \widetilde Z_1,a_2 \widetilde Z_2)),\]
 as both sides are equal to $a_1+a_2$.
\end{proof}

\begin{proof}[Proof of Corollary~\ref{cor:PQD-boundaries}]
    In view of \eqref{eq:trivialLO} and Theorem~\ref{thm:OO-characterization} b), if suffices to investigate the upper and lower bounds of
    $\EE(\min_{i \in A}(a_i Z_i))$ for ${\bm a} \in (0,\infty)^d$ and $A \subset \{1,\dots,d\}$, $A \neq \emptyset$, where ${\bm Z}$ is a generator for $G$. We have
    \begin{align*}
         \EE\big(\min_{i \in A}(a_i Z_i)\big) &\leq \min_{i \in A}(\EE(a_i Z_i))=\min_{i \in A}(a_i)
         \qquad \text{and} \qquad
     \EE\big(\min_{i \in A}(a_i Z_i)\big) \geq \begin{cases}
         a_j & \text{if } A=\{j\},\\
         0 & \text{else},
     \end{cases} 
    \end{align*}
    and the upper and lower  bounds are attained by generators of the fully dependent model (${\bm Z}$ being almost surely ${\bm e}=(1,1,\dots,1)^\top$) and the independent model (${\bm Z}$ being uniformly distributed among the set $\{d{\bm e_1},d{\bm e_2},\dots,d{\bm e}_d\}$), respectively, which implies the assertion.
\end{proof}

\begin{proof}[Proof of Corollary~\ref{cor:PQD-associated-Choquet}]
    The lower orthant order ${\bm X}^* \, \geq_{\mathrm{lo}}\,  {\bm X}$ is known from Theorem~\ref{thm:choquet}. Let 
    ${\bm Z}$ and ${\bm Z}^*$ be generators of the respective models. Since they share identical extremal coefficients, they also share identical tail dependence coefficients $\chi(A)=\EE(\min_{i \in A} Z_i)=\EE(\min_{i \in A} Z^*_i)$, $A \subset \{1,\dots,d\}$, $A \neq \emptyset$, which can be retrieved from $\theta$ via \eqref{eq:chi-theta}. 
    In general, we have for $A \subset \{1,\dots,d\}$, $A \neq \emptyset$, ${\bm a} \in (0,\infty)^d$
    \[
    \EE\big(\min_{i \in A} (a_iZ_i)\big) \geq \min_{i \in A} (a_i) \EE\big(\min_{i \in A} (Z_i)\big) =  \min_{i \in A} (a_i) \cdot \chi(A).
    \]
    The Choquet model attains the lower bound, since with \eqref{eq:Choquet-spectral} and \eqref{eq:chi-from-tau}
     \begin{align*}
     \EE\big(\min_{i \in A} (a_iZ^*_i)\big)&=  \sum_{L \subset \{1,\dots,d\}, L \neq \{\emptyset\}} \tau(L)  \min_{i \in A} (a_i({\bm e}_L)_i)\\
     &=  \sum_{L \subset \{1,\dots,d\}, A \subset L}  \tau(L)  \min_{i \in A} (a_i) = \min_{i \in A} (a_i) \cdot \chi(A).
     \end{align*}  
     So by Theorem~\ref{thm:OO-characterization} we also have ${\bm X}^* \, \leq_{\mathrm{uo}}\,  {\bm X}$, hence the assertion.
\end{proof}

\begin{proof}[Proof of Lemma~\ref{lemma:Choquet-LO-UO}]
The LO part is immediate from $\theta \leq \widetilde \theta$ implying the inclusion of associated max-zonoids $K^* \subset \widetilde K^*$ or Choquet integrals $\ell^* \leq \widetilde \ell^*$ (cf.~Theorem~\ref{thm:choquet}) and then follows directly from Theorem~\ref{thm:OO-characterization} part a). For the UO part, note from the Proof of Corollary~\ref{cor:PQD-associated-Choquet} that for $A \subset \{1,\dots,d\}$, $A \neq \emptyset$, ${\bm a} \in (0,\infty)^d$
\begin{align*}
     \EE\big(\min_{i \in A} (a_iZ^*_i)\big)= \min_{i \in A} (a_i) \cdot \chi(A) \qquad \text{and} \qquad   \EE\big(\min_{i \in A} (a_i\widetilde Z^*_i)\big)= \min_{i \in A} (a_i) \cdot \widetilde \chi(A)
\end{align*}  
if ${\bm Z}$ and $\widetilde {\bm Z}$ are generators of the respective models, hence the assertion with Theorem~\ref{thm:OO-characterization} part b).
\end{proof}

\subsection{Proofs concerning the Dirichlet and H\"usler-Rei{\ss} models}
\label{app:DirichletHR}

\begin{proof}[Proof of Theorem~\ref{thmdef:Dirichlet}]
The equivalence of (ii) and (iii) has been verified in \citet{colestawn91}.
The equivalence of (i) and (ii) follows similarly to \citet{afz_15} (3) from
the fact that ${\bm D}$ is distributed like ${\bm \Gamma}/\lVert {\bm \Gamma}\lVert_1$ and the independence of  ${\bm \Gamma}/\lVert {\bm \Gamma} \rVert_1$ and $\lVert {\bm \Gamma} \rVert_1$. More precisely, let $\ell_1$ and $\ell_2$ be the stable tail dependence functions that arise from the generators (i) and (ii), respectively.
Then $\ell_1$ and $\ell_2$ can be expressed as follows for any $ {\bm x} \in [0,\infty)^d$
\begin{align*}
\ell_1({\bm x}) &= \EE \max_{i=1,\dots,d} \frac{x_i \Gamma_i} {\alpha_i}
= {\EE \lVert {\bm \Gamma} \rVert_1} \cdot \EE \max_{i=1,\dots,d} \frac{x_i \Gamma_i/\lVert {\bm \Gamma} \rVert_1} {\alpha_i},
\\
\ell_2({\bm x}) &=  \EE \max_{i=1,\dots,d} \frac{x_i \lVert{\bm\alpha}\rVert_1 D_i} {\alpha_i} =  \lVert{\bm\alpha}\rVert_1 \cdot \EE \max_{i=1,\dots,d} \frac{x_i D_i} {\alpha_i}.
\end{align*}
If suffices to note $\EE \lVert {\bm \Gamma} \rVert_1 = \lVert{\bm\alpha}\rVert_1$ in order to conclude $\ell_1=\ell_2$.
\end{proof}

In order to prove Theorem~\ref{thm:Dirichlet-PQD} we will use a simple inequality that follows from the theory of majorisation \citep{moa10}.

\begin{proposition}[\citet{marshallproschan65} Corollary~3, \citet{moa10} Proposition~B.2.b.] \label{prop:from-majorisation}
    Let $g:\RR \to \RR$ be continuous and convex and let $X_1,X_2,\dots$ be a sequence of independent and identically distributed random variables, then \[
    \EE \, g \bigg(\sum_{i=1}^n \frac{X_i}{n}\bigg)
    \]
    is nonincreasing in $n=1,2,\dots$.
\end{proposition}

\begin{corollary}\label{cor:GammaMonotonicity}
     Let $g:\RR \to \RR$ be continuous and convex and, let $Z^{(\alpha)} \sim \Gamma(\alpha)$ follow a univariate Gamma distribution with shape parameter $\alpha>0$, then
     \[
    \EE \, g \bigg(\frac{Z^{(\alpha)}}{\alpha}\bigg)
    \]
    is nonincreasing in $\alpha \in (0,\infty)$.
\end{corollary}

\begin{proof}
    We consider first the case that $\alpha=(k/n) \cdot \beta$ for some natural numbers $1\leq  k < n$. Then consider independent and identically distributed random variables $\Gamma_1,\Gamma_2,\dots$ following a $\Gamma(\beta/n)$ distribution. Then Proposition~\ref{prop:from-majorisation} gives
    \[
     \EE \, g \bigg(\frac{\sum_{i=1}^k  \Gamma_i}{\alpha}\bigg)=
       \EE \, g \bigg(\frac{n}{\beta} \cdot \sum_{i=1}^k \frac{\Gamma_i}{k}\bigg) \geq 
       \EE \, g \bigg(\frac{n}{\beta} \cdot \sum_{i=1}^n \frac{\Gamma_i}{n}\bigg)
       = \EE \, g \bigg(\frac{\sum_{i=1}^n  \Gamma_i}{\beta}\bigg).
    \]
    By the convolution stability of the Gamma distribution
    \[
    \sum_{i=1}^k \Gamma_i \sim \Gamma(\alpha) 
    \qquad \text{and} \qquad
     \sum_{i=1}^n \Gamma_i \sim \Gamma(\beta).
    \]
    Hence, the assertion is shown for $\alpha$ and $\beta$ that differ by a rational multiplier.
    
    If we only know $\alpha < \beta$, consider a decreasing sequence $\beta_n \downarrow \beta$, such that $\alpha$ and $\beta_n$ differ by a rational multiplier. This gives
$\EE g({Z^{(\alpha)}}/{\alpha}) \geq \limsup_{n \to \infty} \EE g({Z^{(\beta_n)}}/{\beta_n})$ by the above argument. On the other hand, Fatou's lemma gives $\EE g({Z^{(\beta)}}/{\beta}) \leq \liminf_{n \to \infty} \EE g({Z^{(\beta_n)}}/{\beta_n})$. This finishes the proof.
\end{proof}

\begin{proof}[Proof of Theorem~\ref{thm:Dirichlet-PQD}]
If ${\bm \alpha}={\bm \beta}$, the statement is clear. Else,
because the parameter space of the Dirichlet model is $(0,\infty)^d$, we can find a chain of parameter vectors ${\bm \alpha}={\bm \alpha}^{(0)}\leq {\bm \alpha}^{(1)}\leq \dots \leq {\bm \alpha}^{(m)}={\bm \beta}$, such that for each $i=0,\dots,m-1$, the vectors ${\bm \alpha}^{(i)}$ and ${\bm \alpha}^{(i+1)}$ differ only by one component. 
Hence it suffices to consider the case, where  ${\bm \alpha}$ and ${\bm \beta}$ differ only in one component. Without loss of generality, let this be the first component. 

Let ${\bm Z}$ be a Gamma generator for $\mathrm{MaxDir}({\bm \alpha})$ and $\widetilde {\bm Z}$ be a Gamma generator for for $\mathrm{MaxDir}({\bm \beta})$ in the sense of 
Theorem/Definition~\ref{thmdef:Dirichlet}. 
Then we may assume that $Z_i=\widetilde Z_i$ for $i=2,\dots,d$, whereas $\alpha_1Z_1\sim \Gamma(\alpha_1)$ and $\beta_1 \widetilde Z_1 \sim \Gamma(\beta_1)$ are independent from $(Z_2,\dots,Z_d)^\top$, and $\alpha_1<\beta_1$ by assumption. 
We will need to show (cf.~Theorem~\ref{thm:OO-characterization}) that for fixed ${\bm a} \in (0,\infty)^d$ and $A \subset \{1,\dots,d\}$, $A \neq \emptyset$
\[
\EE \min_{i \in A} \big(a_iZ_i\big) \leq \EE \min_{i \in A} \big(a_i\widetilde Z_i\big)
\qquad \text{and} \qquad 
\EE \max_{i=1,\dots,d} \big(a_iZ_i\big) \geq \EE \max_{i=1,\dots,d} \big(a_i\widetilde Z_i\big).
\]
Due to the setting above, it suffices to consider only subsets $A$ with $1 \in A$,
and due to the marginal standardisation $\EE(Z_1)=1$,  it suffices to restrict our attention to $A \setminus \{1\} \neq \emptyset$. Setting $V_A= \min_{i \in A \setminus\{1\}}{(a_iZ_i)}$ and $W= \max_{i=2,\dots,d}{(a_iZ_i)}$ this means the assertion will follow from
\[
\EE \min \big(a_1Z_1, V_A\big) \leq \EE \min\big(a_1\widetilde Z_1,V_A\big)
\qquad \text{and} \qquad 
\EE \max\big(a_1Z_1,W\big) \geq \EE \max\big(a_1\widetilde Z_1,W\big).
\]
Indeed, this is implied by Corollary~\ref{cor:GammaMonotonicity}, when considering the continuous convex functions $g_c(x)=-\min(a_1 x,c)$ or $g_c(x)=\max(a_1 x,c)$ for a constant $c\in \RR$.
\end{proof}

\begin{proof}[Proof of Theorem~\ref{thm:HRm}] Set 
\[\rho_{ij}^{(n)}=\exp(-\gamma_{ij}/(4 \log(n)) \qquad \text{and}\qquad  \widetilde \rho_{ij}^{(n)}=\exp(-\widetilde \gamma_{ij}/(4 \log(n))\]
for $i,j \in \{1,\dots,d\}$, $n \in \NN$, so that ${\bm \gamma}, {\widetilde {\bm \gamma}} \in \mathcal{G}_d$ ensures that the resulting matrices are correlation matrices, cf.~e.g.~\citet{bcr84} Theorem~3.2.2.
By construction, $\rho_{ij}^{(n)} \geq \widetilde\rho_{ij}^{(n)}$ for all $i,j,n$. 
And so the normal comparison lemma \citep{slepian_62}, cf. also \citet{tong_80} Section 2.1.\ or \citet{mueller_02} Example~3.8.6, implies that ${\bm Y} \geq_{\mathrm{PQD}} \widetilde{\bm Y}$ if ${\bm Y}$ and $\widetilde {\bm Y}$ are zero mean unit-variance Gaussian random vectors with correlations ${\bm \rho}$ and $\widetilde {\bm \rho}$, respectively. In fact, even ${\bm Y} \geq_{\mathrm{sm}} \widetilde{\bm Y}$ holds for the supermodular order \citep[Theorem 3.13.5]{mueller_02}.

Consider the triangular arrays with independent ${\bm Y}^{(n)}_i \sim {\bm Y}$, $i=1,\dots,n$ and $\widetilde {\bm Y}^{(n)}_i \sim \widetilde {\bm Y}$, $i=1,\dots,n$. Since $t(1-\exp(-a/t)) \to a$ as $t \to \infty$, Theorem~\ref{thm:triangularArrayHR} gives that
\[
u_n ({\bm M}^{(n)}-u_n) = u_n \cdot \bigg( \max_{i=1,\dots,n} ({\bm Y}^{(n)}_i)_1 - u_n ,\dots,  \max_{i=1,\dots,n} ({\bm Y}^{(n)}_i)_d - u_n  \bigg)^\top
\]
converges in distribution to $\mathrm{HR}({\bm \gamma})$ and the corresponding tilde-version, while the closure under independent conjunction (\citet{shsh_07} Theorem 9.A.5) together with \citet{shsh_07} Theorem~9.A.4 implies $u_n ({\bm M}^{(n)}-u_n) \geq_{\mathrm{PQD}} u_n (\widetilde {\bm M}^{(n)}-u_n)$ for all $n \in \NN$. 
In fact, even  $u_n ({\bm M}^{(n)}-u_n) \geq_{\mathrm{sm}} u_n (\widetilde {\bm M}^{(n)}-u_n)$ for all $n \in \NN$ as the supermodular order is also closed under independent conjuction \citep[Theorem~3.9.14]{mueller_02} and note \citet{shsh_07} Theorem~9.A.12.
The assertion follows now from the closure of the PQD-order under distributional limits (\citet{shsh_07} Theorem 9.A.5).
We even have $\mathrm{HR}({\bm \gamma}) \geq_{\mathrm{sm}} \mathrm{HR}(\widetilde {\bm \gamma})$, as the supermodular order satisfies the same closure property with respect to distributional limits \citep[Theorem~3.9.12]{mueller_02}.
\end{proof}

\section{Complete alternation and Bernstein functions}\label{app:harmonic}

We recall some elementary definitions and facts from \citet{bcr84}, cf.~also \citet{molchanovTheory}. 
Let $(S,\circ,e)$ be an \emph{abelian semigroup}, that is, a non-empty set $S$ with a composition $\circ$ that is associative and commutative and has a neutral element $e$. Three examples are of interest to us: 
\begin{enumerate}[(i)]
\itemsep0mm
    \item $S=[0,\infty)$ with $+$ and neutral element $0$, 
    \item $S=\mathcal{P}_d$, the power set of $\{1,\dots,d\}$, with the union operation $\cup$ and neutral element $\emptyset$,  
    \item $S=[0,\infty)^d$ with the componentwise maximum operation $\vee$ and neutral element ${\bm 0}$.
\end{enumerate}
Examples (ii) and (iii) are even \emph{idempotent} semigroups, as $s \circ s = s$ for these operations. 
We use the standard notation
\[(\Delta_{b} f)(a) = f(a) - f(a \circ b).\]
\begin{definition}\label{def:CA}
  A function $f: S \to \RR$ is called \emph{completely alternating} if for all $n\geq 1$, $\{s_1,\dots,s_n\}\subset S$ and $s\in S$ 
\[
(\Delta_{s_1} \Delta_{s_2} \dots \Delta_{s_n} f) (s) = \sum_{I \subset \{1,\dots,n\}} (-1)^{|I|} f( s \circ \bigcirc_{i \in I} \, s_i) \leq 0.
\]  
\end{definition}
For idempotent semigroups (examples (ii) and (iii) above), the complete alternation property coincides with negative definiteness, cf.~\citet{bcr84} 4.4.16 and 4.6.8.
\begin{definition}
  A function $f: S \to \RR$ is called \emph{negative definite} if for all $n\geq 2$, $\{s_1,\dots,s_n\}\subset S$, $\{a_1,\dots,a_n\}\subset \RR$ with $a_1+\dots+a_n=0$ 
\[
\sum_{j=1}^n \sum_{k=1}^n a_j a_k f (s_j \circ s_k) \leq 0.
\]  
\end{definition}

In the context of multivariate extremes, max-complete alternation of the stable tail dependence function implies union-complete alternation of the extremal coefficient function. In fact, the following directional version holds true irrespective of whether we take homogeneity or marginal standardisation into account or not.

\begin{lemma} \label{lem:ell-maxCA-theta-unionCA}
    Let $\ell:[0,\infty)^d \to [0,\infty)$ be max-completely alternating. Let ${\bm x} \in [0,\infty)^d$. Let $\theta^{(x)}: \mathcal{P}_d \to [0,\infty)$ be defined as $\theta^{(x)}(A)=\ell({\bm x}_A)$, where ${\bm x}_A={\bm x} \cdot {\bm e}_A \in [0,\infty)^d$ is the vector with $x_i$ as $i$-th coordinate if $i \in A$ and zero else. Then $\theta^{(x)}$ is union-completely alternating.
\end{lemma}

\begin{proof}
    The result follows from the observation that ${\bm x}_{A \cup B}={\bm x}_A \vee {\bm x}_B$ for $A,B \in \mathcal{P}_d$. Therefore,
    \[(\Delta_{A_1} \dots \Delta_{A_n} \theta^{(x)}) (A) = (\Delta_{{\bm x}_{A_1}}  \dots \Delta_{{\bm x}_{A_n}} \ell) ({\bm x}_A) \leq 0  \]
    for $A,A_1,\dots,A_n \in \mathcal{P}_d$, where $n \geq 1$.
\end{proof}

\begin{lemma}[Independent concatenation] Let $\theta_1:\mathcal{P}(M) \to [0,\infty)$ and $\theta_2:\mathcal{P}(N) \to [0,\infty)$ be union-completely alternating, where $\mathcal{P}(M)$ and $\mathcal{P}(N)$ are the power sets of finite sets $M$ and $N$, respectively, such that $\theta_1(\emptyset)=\theta_2(\emptyset)=0$.
Then $\theta:\mathcal{P}(M \cup N) \to [0,\infty)$ with $\theta(A)=\theta_1(A \cap M) + \theta_2(A \cap N)$ is union-completely alternating and $\theta(\emptyset)=0$.
\end{lemma}

\begin{proof}
    By the Choquet theorem \citep[Theorem 2.3.2]{schneiderweil2008} we may express  \[\theta_1(A)=\sum_{K \in \mathcal{P}(M) : K \cap A \neq \emptyset} a_{K}
    \qquad \text{ and } \qquad \theta_2(B)=\sum_{L \in \mathcal{P}(N) : L \cap B \neq \emptyset} b_{L}
    \]
    for non-negative coefficients $a_{K}$, $K \subset M$, $K \neq \emptyset$ and $b_L$, $L \subset N$, $L \neq \emptyset$. Define for $A \subset M$, $B \subset N$
    \[\theta(A \cup B)=\sum_{(K,L) \in \mathcal{P}(M)\times \mathcal{P}(N) : (K\cup L) \cap (A \cup B) \neq \emptyset} c_{K \cup L},
    \quad \text{ where } \quad c_{K \cup L}=\begin{cases}
        a_K & \text{if } K \neq \emptyset, L=\emptyset,\\
        b_L & \text{if } K = \emptyset, L\neq \emptyset,\\
        0 & \text{if } K \neq \emptyset, L \neq \emptyset.
    \end{cases}
    \]
    Then it is easily seen that $\theta(A \cup B)=\theta_1(A) + \theta_2(B)$, hence the assertion.
\end{proof}

\begin{corollary}\label{cor:theta-augmented}
    Let $\theta:\mathcal{P}_d \to [0,\infty)$ be union-completely alternating with $\theta(\emptyset)=0$. Then 
    $\theta' : \mathcal{P}_{d+1} \to [0,\infty)$, $\theta'(A)=\theta(A \cap \{1,\dots,d\})+ c \mathbf{1}_{d+1 \in A}$ is
    union-completely alternating with $\theta'(\emptyset)=0$ for any $c\geq 0$.
\end{corollary}

There are various equivalent definitions for \emph{Bernstein functions}. For us it will be sufficient to consider the following. The equivalence of (i) and (ii) in the following theorem is a consequence from the 2-divisibility of $([0,\infty), +, 0)$, cf.~\citet{bcr84} 4.6.8.

\begin{thmdef} A function $g:[0,\infty) \to \RR$ is called a \emph{Bernstein function}  if one of the following equivalent conditions is satisfied:
\begin{enumerate}[(i)]
\itemsep0mm
\item $g \geq 0$, $g$ is continuous, and $g$ is negative definite with respect to addition.
\item $g \geq 0$, $g$ is continuous, and
$g$ is completely alternating with respect to addition.
    \item 
$g$ can be expressed as
\[g(r)=a+br+\int_0^{\infty}(1-\mathrm{e}^{-tr}) \nu (\mathrm{d}t), \quad r\geq0,\] where $a,b\geq0$ and $\nu$ is a non-negative Radon measure on $(0,\infty)$ satisfying the integrability condition $\int_0^{\infty} \min(t,1)\nu (\mathrm{d}t)<\infty$. 
\end{enumerate}
\end{thmdef}

An important property of Berstein functions is that they act on negative definite kernels with non-negative diagonal, cf.~\citet{bcr84} 4.4.3. 

\begin{corollary} \label{cor:BernsteinAction}
     Let $S$ be an idempotent semigroup and $f: S \rightarrow [0,\infty)$ be completely alternating and $g$ a Bernstein function. Then the composition map $g \circ f: S \to [0,\infty)$ is completely alternating.  
\end{corollary}

\begin{corollary} \label{cor:augmented-inequality}
   Let $\theta:\mathcal{P}_d \to [0,\infty)$ be union-completely alternating with $\theta(\emptyset)=0$ and $g$ be a Bernstein function. Let $A^* \subset \{1,\dots,d\}$ and $c>0$. Then
      \[
    \sum_{J \subset \{1,\dots,d\} \setminus A^*} (-1)^{|J|} g\big(\theta(A^* \cup J)\big)
    \leq 
      \sum_{J \subset \{1,\dots,d\} \setminus A^*} (-1)^{|J|} g\big(\theta(A^* \cup J)+c\big).
    \]
\end{corollary}

\begin{proof}
    By Corollary~\ref{cor:theta-augmented}, the function
       $\theta' : \mathcal{P}_{d+1} \to [0,\infty)$, $\theta'(A)=\theta(A \cap \{1,\dots,d\})+ c \mathbf{1}_{d+1 \in A}$ is
    union-completely alternating with $\theta'(\emptyset)=0$. Hence, Corollary~\ref{cor:BernsteinAction} implies that $g \circ \theta'$ is again union-completely alternating.
    Hence, by Definition~\ref{def:CA} and since $\{1,\dots,d,d+1\} \setminus A^*$ is not empty (it contains at least the element $d+1$)
     \[
    \sum_{J' \subset \{1,\dots,d,d+1\} \setminus A^*} (-1)^{|J'|} g\big(\theta'(A^* \cup J')\big)
    \leq 0.
    \]
    Now each $J'$ above is either a subset $J$ of $\{1,\dots,d\} \setminus A^*$ or it is of the form $J \cup \{d+1\}$, where $J$ is a subset of $\{1,\dots,d\} \setminus A^*$. Separating the summands accordingly gives the assertion.
\end{proof}

The following proposition is the key argument to establish the implication $\Lambda \leq_{\mathrm{uo}} \widetilde \Lambda$ $\Rightarrow$ $G \leq_{\mathrm{uo}} \widetilde G$ 
in Theorem~\ref{thm:OO-characterization}.

\begin{proposition} \label{prop:key-OO-fundamentals}
    Let $\theta: \mathcal{P}_d \rightarrow [0,\infty)$ and $\widetilde \theta: \mathcal{P}_d \rightarrow [0,\infty)$ be union-completely alternating with $\theta(\emptyset)=\widetilde \theta(\emptyset)=0$. For $A \subset \{1,\dots,d\}$, $A \neq \emptyset$ set
    \[\chi(A)=\sum_{I \subset A, \, I \neq \emptyset} (-1)^{|I|+1} \theta(I) \qquad \text{and} \qquad \widetilde \chi(A)=\sum_{I \subset A, \, I \neq \emptyset} (-1)^{|I|+1} \widetilde \theta(I).\]
    Suppose
    \[ \chi(A) \leq \widetilde \chi(A) \quad \text{ for all }  \quad A \subset \{1,\dots,d\}, \, A \neq \emptyset. \]
    Let $g: [0,\infty) \to [0,\infty)$ be a Bernstein function. Then 
    \[
    \sum_{I \subset \{1,\dots,d\}} (-1)^{|I|+1} g\big(\theta(I)\big)
    \leq \sum_{I \subset \{1,\dots,d\}} (-1)^{|I|+1} g\big(\widetilde \theta(I)\big).
    \]
\end{proposition}

\begin{remark}
Under the assumptions of Proposition~\ref{prop:key-OO-fundamentals} we have also
        \[
    \sum_{I \subset A} (-1)^{|I|+1} g\big(\theta(I)\big)
    \leq \sum_{I \subset A} (-1)^{|I|+1} g\big(\widetilde \theta(I)\big)
    \]
    for any non-empty subset $A$ of $\{1,\dots,d\}$.
This follows directly from the proposition as we may restrict $\theta$ and $\widetilde \theta$ to the respective subset $A$ and all assumptions that were previously made for $\{1,\dots,d\}$ will be valid for the restrictions to $A$, too. 
\end{remark}
\begin{proof} The inverse linear operation to recover $\theta$ from $\chi$ is given by
\[
\theta(A)=\sum_{I \subset A, \, I \neq \emptyset} (-1)^{|I|+1} \chi(I)
\]
(and likewise for $\widetilde \theta$ and $\widetilde \chi$), so that both quantities contain the same information. If $\chi=\widetilde \chi$ and hence $\theta=\widetilde \theta$, the statement is trivially true.
Otherwise, we will show the proposition in two steps. First, we will establish its validity in the situation when $\chi(A) < \widetilde \chi(A)$ only for one $A^* \subset \{1,\dots,d\}$, $A^* \neq \emptyset$ and $\chi(A) = \widetilde \chi(A)$ for all other $A \subset \{1,\dots,d\}$, $A \neq \emptyset$.  Second, we will show how this allows us to derive the proposition using convexity and continuity arguments.

\textit{Step 1:}
Let $\chi(A) < \widetilde \chi(A)$ only for one $A^* \subset \{1,\dots,d\}$, $A^* \neq \emptyset$ and $\chi(A) = \widetilde \chi(A)$ for all other $A \subset \{1,\dots,d\}$, $A \neq \emptyset$. Then $c=\widetilde \chi(A^*) - \chi(A^*)>0$ and 
\begin{align*}
    \widetilde \theta(A)=\begin{cases}
        \theta(A) + c & \text{if } A^* \subset A,\\
        \theta(A) & \text{else}
    \end{cases}
\end{align*}
if $|A^*|$ is odd, and
\begin{align*}
    \theta(A)=\begin{cases}
        \widetilde \theta(A) + c & \text{if } A^* \subset A,\\
        \widetilde \theta(A) & \text{else}
    \end{cases}
\end{align*}
if $|A^*|$ is even, and in both situations it suffices to show that
    \[
    \sum_{I \subset \{1,\dots,d\} : A^* \subset I} (-1)^{|I|+1} g\big(\theta(I)\big)
    \leq \sum_{I \subset \{1,\dots,d\}: A^* \subset I} (-1)^{|I|+1} g\big(\widetilde \theta(I)\big),
    \]
    which is equivalent to
    \[
    \sum_{J \subset \{1,\dots,d\}\setminus A^*} (-1)^{|J|+|A^*|+1} g\big(\theta(A^* \cup J)\big)
    \leq \sum_{J\subset \{1,\dots,d\}\setminus A^*} (-1)^{|J|+|A^*|+1} g\big(\widetilde \theta(A^* \cup J)\big).
    \]
    Hence, if $|A^*|$ is odd, we need to establish
        \[
    \sum_{J \subset \{1,\dots,d\}\setminus A^*} (-1)^{|J|} g\big(\theta(A^* \cup J)\big)
    \leq \sum_{J\subset \{1,\dots,d\}\setminus A^*} (-1)^{|J|} g\big(\theta(A^* \cup J)+c\big),
    \]
    and, if $|A^*|$ is even, 
    we need to establish
        \[
    \sum_{J \subset \{1,\dots,d\}\setminus A^*} (-1)^{|J|} g\big(\widetilde \theta(A^* \cup J)\big)
    \leq \sum_{J\subset \{1,\dots,d\}\setminus A^*} (-1)^{|J|} g\big(\widetilde \theta(A^* \cup J)+c\big).
    \]
    Both inequalities now follow directly from Corollary~\ref{cor:augmented-inequality}.

    \textit{Step 2:} Let $C_d$ be the set of points $x=(x_A)_{A \in \mathcal{P}_d \setminus \{\emptyset\}}$ in $\RR^{\mathcal{P}_d \setminus \{\emptyset\}}$ such that the mapping $A \mapsto x_A$ becomes union-completely alternating when setting $x_{\emptyset}=0$. Then $C_d$ is a convex cone  with non-empty interior and  
    $C_d \subset [0,\infty)^{\mathcal{P}_d \setminus \{\emptyset\}}$ with $(0,0,\dots,0) \in C_d$. Let $T:\RR^{\mathcal{P}_d \setminus \{\emptyset\}} \to \RR^{\mathcal{P}_d \setminus \{\emptyset\}}$ be the linear map, such that
    \[(Tx)_A = \sum_{I \subset A, \, I \neq \emptyset} (-1)^{|I|+1} x_I. \]
    Then $T \circ T$ is the identity mapping, hence $T$ is invertible. In particular $D_d=\{Tx \,:\, x \in C_d\}$ is also a convex cone with non-empty interior and $C_d=\{Tx \,:\, x \in D_d\}$.
    We also note that  $D_d \subset [0,\infty)^{\mathcal{P}_d \setminus \{\emptyset\}}$, cf.~\eqref{eq:chi-nonnegative} and that $(0,0,\dots,0) \in D_d$.
    Within the setting of the proposition, we have $\theta, \widetilde \theta \in C_d$ and $\chi,\widetilde \chi \in D_d$ with $\theta=T(\chi)$, $\widetilde \theta= T (\widetilde \chi)$ and $\chi=T(\theta)$, $\widetilde \chi= T (\widetilde \theta)$. 
    
    If both $\theta$ and $\widetilde \theta$ are points in the interior of $C_d$, then $\chi$ and $\widetilde \chi$ are in the interior of $D_d$. Therefore, there exists $\varepsilon>0$ such that the Minkowski sum of the line segment between $\chi$ and $\widetilde \chi$ and an (e.g. Euclidean) $\varepsilon$-ball centered at $(0,0,\dots,0) \in \RR^{\mathcal{P}_d \setminus \{\emptyset\}}$ is completely contained in $D_d$. Within this set we can find a chain $\chi=\chi^{(0)}\leq \chi^{(1)} \leq \chi^{(2)} \leq \dots \leq \chi^{(n)}=\widetilde \chi$, such that for each $i=0,\dots,n-1$ we have that $\chi^{(i)}$ and  $\chi^{(i+1)}$ differ only in one component. By construction, we also have that
    $\theta^{(i)}=T(\chi^{(i)})\in C_d$ and $\theta^{(i+1)}=T(\chi^{(i+1)})\in C_d$, so that we are in the situation of \textit{Step 1} and we may conclude that 
        \[
    \sum_{I \subset \{1,\dots,d\}} (-1)^{|I|+1} g\big(\theta^{(i)}(I)\big)
    \leq \sum_{I \subset \{1,\dots,d\}} (-1)^{|I|+1} g\big(\theta^{(i+1)}(I)\big)
    \]
    for all $i=0,\dots,n-1$, hence the assertion (which does not depend on the choice of $\varepsilon$ or the choice of the chain).
    In other words, we have established the assertion of the proposition if both $\theta$ and $\widetilde \theta$ are points in the interior of $C_d$.

    To complete the argument,  note that 
    the mapping $f:C_d \to \RR$ with
    \[   f(x) = g(0) + \sum_{I \subset \{1,\dots,d\}, I \neq \emptyset} (-1)^{|I|+1} g\big(x_I\big)\]
    is continuous.
    Let $v \in C_d$ be a vector in the interior of $C_d$. Then, for any $\delta>0$ both $\theta+\delta v$ and $\widetilde \theta + \delta v$ are in the interior of $C_d$, whereas $\chi + \delta T(v) = T(\theta + \delta v)$ and $\widetilde \chi + \delta T(v)=T(\widetilde \theta + \delta v)$ are in the interior of $D_d$ and still satisfy $\chi + \delta T(v) \leq \widetilde \chi + \delta T(v)$. Therefore, $f(\theta+\delta v)\leq f(\widetilde \theta + \delta v)$. Finally, since $f$ is continuous, we can find for given $\varepsilon>0$ a corresponding $\delta>0$, such that $f(\theta+\delta v)$ is $\varepsilon$-close to $f(\theta)$, while $f(\widetilde \theta+\delta v)$ is $\varepsilon$-close to $f(\widetilde \theta)$. The assertion of the proposition $f(\theta)\leq f(\widetilde \theta)$ follows as we may choose $\varepsilon$ arbitrarily close to zero.
\end{proof}

\section{Calculation of the max-zonoid envelope} \label{app:envelope}

Let $K$ be the max-zonoid (or dependency set) associated with a stable tail dependence function $\ell$ of a simple max-stable random vector, that is, 
\[K = \big\{{\bm k} \in [0,\infty)^d  \,:\, \langle {\bm k}, {\bm u} \rangle \leq \ell({\bm u}) \quad \text{for all } {\bm u} \in \Sph^{(d-1)}_+ \big\}, \]
and, conversely,
\[\ell({\bm x}) = \sup\{ \langle {\bm x}, {\bm k}\rangle \,:\, {\bm k} \in K\}, \quad {\bm x} \in [0,\infty)^d, \]
cf. \citet{molchanov_08}.
Here, $\Sph^{(d-1)}_+=\{{\bm u} \in [0,\infty)^d \;:\; \lVert {\bm u} \rVert_2 = 1 \}$ denotes the $(d-1)$-dimensional Euclidean unit sphere in $\RR^d$ restricted to the upper orthant $[0,\infty)^d$.
It is well-known that
\[ \Delta^d \subset K \subset [0,1]^d, \]
where the cross-polytope $\Delta^d = \{{\bm x} \in [0,\infty)^d \,:\, \langle{\bm x},{\bm 1}\rangle \leq 1\}$ corresponds to perfect dependence, whereas the cube $[0,1]^d$ corresponds to independence. In particular, in the direction along the $i$-th axis the set $K$ contains precisely the set $\{t {\bm e_i} \,:\, t \in [0,1]\}$.

For illustrative purposes we restrict our attention to $d=2$, where we seek to calculate a parametrisation of the boundary curve of a general dependency set $K$. 
To this end, we parametrise the upper unit circle via ${\bm u}=(\cos(\alpha),\sin(\alpha))^T \in \Sph^1_+$ for $\alpha \in [0,\pi/2]$ and we assume that  $\ell$ is differentiable.
For $\alpha \in (0,\pi/2)$ a point $(x_1,x_2)$ on the desired envelope curve will then satisfy the two conditions 
\begin{align*}
    \Big \langle \begin{pmatrix} \cos(\alpha) \\ \sin(\alpha) \end{pmatrix}, \begin{pmatrix} x_1 \\ x_2 \end{pmatrix} \Big \rangle - \ell  \begin{pmatrix} \cos(\alpha) \\ \sin(\alpha) \end{pmatrix}  &= 0,\\
      \frac \partial {\partial \alpha}  \bigg[ \Big \langle \begin{pmatrix} \cos(\alpha) \\ \sin(\alpha) \end{pmatrix}, \begin{pmatrix} x_1 \\ x_2 \end{pmatrix} \Big \rangle - \ell  \begin{pmatrix} \cos(\alpha) \\ \sin(\alpha) \end{pmatrix}\bigg] &= 0,
\end{align*}
which can be seen by a standard calculus of variations argument \citep{EncyMath}.
Let $\partial_1 \ell$ and $\partial_2 \ell$ denote the partial derivatives of $\ell$ with respect to first and second component. The two conditions can be then be expressed as
\begin{align*}
   x_1 \cos(\alpha) + x_2 \sin(\alpha) &= \ell(\cos(\alpha), \sin(\alpha)) \\
    -x_1 \sin(\alpha) + x_2 \cos(\alpha) &= -\sin(\alpha) \partial_1 \ell (\cos(\alpha), \sin(\alpha)) + \cos(\alpha) \partial_2 \ell (\cos(\alpha),\sin(\alpha)) .
\end{align*}
Solving the system for $x_1$ and $x_2$ (while taking into account $\sin^2(\alpha)+\cos^2(\alpha)=1$) gives
\begin{align}
\label{eq:x1}
 x_1 &= \cos(\alpha)  L(\alpha) +  \sin^2(\alpha)  L_1(\alpha) - \sin(\alpha)\cos(\alpha) L_2(\alpha),\\
 \label{eq:x2}
    x_2 &= \sin(\alpha)  L(\alpha) -  \sin(\alpha) \cos(\alpha) L_1(\alpha) + \cos^2(\alpha) L_2(\alpha),
\end{align}
where 
\begin{align*}
    L(\alpha)=\ell(\cos(\alpha), \sin(\alpha))
    \qquad \text{ and } \qquad
    L_i(\alpha)= \partial_i \ell(\cos(\alpha),\sin(\alpha)), \quad i=1,2.
\end{align*}
The parametrisation of the boundary curve of $K$ as given by \eqref{eq:x1} and \eqref{eq:x2} is the basis for all our plots in this text.

\begin{example}[H\"usler-Rei{\ss} distribution]
For the bivariate H\"usler-Rei{\ss} family with stable tail dependence function 
\begin{align}\label{eq:ellHR}
 \ell(x_1,x_2)=
     x_1 \Phi\bigg(\frac{\eta}{2} +\frac{\log(x_1/x_2)}{\eta}\bigg)+x_2 \Phi\bigg(\frac{\eta}{2}+\frac{\log(x_2/x_1)}{\eta}\bigg),
 \end{align}
 where $\eta^2=\gamma_{12}$,
straightforward calculations show that
\begin{align*}
    L_1(\alpha)= \widetilde L(\cot(\alpha)) 
    \qquad \text{and} \qquad
    L_2(\alpha)= \widetilde L(\tan(\alpha)),
\end{align*}
with
\begin{align*}
     \widetilde L(t)&=
     \Phi\left(\frac{\eta}{2} +\frac{\log\left(t \right)}{\eta} \right)
+ \frac{1}{\eta}  \varphi\left(\frac{\eta}{2} +\frac{\log\left(t \right)}{\eta} \right) 
-  \frac{1}{\eta t} \varphi\left(\frac{\eta}{2} -\frac{\log\left(t\right)}{\eta} \right).
\end{align*}
\end{example}

In other situations the spectral density $h$ of $\ell$ may be known, such that 
\[\ell(x_1,x_2) = \int_{0}^1 \max( \omega x_1, (1-\omega) x_2) h(\omega) d\omega.
\]

\begin{example}[Dirichlet model]
The spectral density of the bivariate Dirichlet model with parameter vector $(\alpha_1,\alpha_2) \in (0,\infty)^2$ is given by 
\begin{align*}
    h(\omega)= \frac{\Gamma(\alpha_1 + \alpha_2 + 1)}{(\alpha_1 \omega + \alpha_2 (1-\omega))^{(\alpha_1 + \alpha_2 + 1)}} \frac{\alpha_1^{\alpha_1}}{\Gamma(\alpha_1)} \frac{\alpha_2^{\alpha_2}}{\Gamma(\alpha_2)}
    \omega^{\alpha_1-1} (1-\omega)^{\alpha_2-1}.
\end{align*}
\end{example}

Let us abbreviate 
\[H(t)=\int_0^t h(\omega) d\omega \qquad \text{and} \qquad \widetilde H(t) = \int_0^t \omega h(\omega) d\omega.\]
Taking into account the identities $H(1) =2$ and $\widetilde H(1) =1$ (due to marginal standardisation) straightforward calculations yield
\begin{align*}
    \ell(x_1,x_2) &= x_1 - (x_1+x_2) \widetilde H\bigg(\frac{x_2}{x_1+x_2}\bigg) + x_2 H\bigg( \frac{x_2}{x_1+x_2} \bigg),\\
    \partial_1 \ell(x_1,x_2) &= 1 - \widetilde H\bigg(\frac{x_2}{x_1+x_2}\bigg),\\
    \partial_2 \ell(x_1,x_2) &= H\bigg(\frac{x_2}{x_1+x_2}\bigg) 
    - \widetilde H\bigg(\frac{x_2}{x_1+x_2}\bigg).
\end{align*}
Hence,
\begin{align*}
    L(\alpha) &= \cos(\alpha) - (\sin(\alpha)+\cos(\alpha)) \widetilde H\bigg(\frac{1}{1+\cot(\alpha)}\bigg) + \sin(\alpha) H\bigg(\frac{1}{1+\cot(\alpha)}\bigg),\\
    L_1(\alpha) &= 1 - \widetilde H\bigg(\frac{1}{1+\cot(\alpha)}\bigg),\\
    L_2(\alpha) &= H\bigg(\frac{1}{1+\cot(\alpha)}\bigg) 
    - \widetilde H\bigg(\frac{1}{1+\cot(\alpha)}\bigg).
\end{align*}

\end{document}